\documentclass[12pt]{article}
\usepackage{amsmath,amsthm,amssymb,amsfonts}
\usepackage{enumerate}
\usepackage{mathrsfs}
\usepackage{caption}
\usepackage[width=16cm,height=21cm]{geometry}
\usepackage{tikz}
\usepackage{tikz-cd}
\usetikzlibrary{shapes.geometric}
\numberwithin{equation}{section}
\numberwithin{figure}{section}

\newtheorem{theorem}{Theorem}[section]
\newtheorem*{theorem*}{Theorem}
\newtheorem{lemma}[theorem]{Lemma}
\newtheorem{proposition}[theorem]{Proposition}
\newtheorem{corollary}[theorem]{Corollary}

\theoremstyle{definition}

\newtheorem{remark}[theorem]{Remark}

\let\smash=\wedge
\let\iso=\cong

\let\directsum=\oplus
\let\minus=\smallsetminus

\let\units=\times

\newcommand{\bideg}{(\star)}

\newcommand{\unit}{\mathbf{1}}
\newcommand{\SH}{\mathbf{SH}}

\newcommand{\ropen}{\hookrightarrow}
\newcommand{\rclosed}{\hookrightarrow}

\newcommand{\pistablesheaf}{\underline{\pi}}
\newcommand{\hyper}{\mathsf{h}}
\newcommand{\liftofhyper}{\mathsf{h}^\prime}

\newcommand{\Sing}{\mathrm{Sing}^{\mathbf{A}^1}}

\newcommand{\hofib}{\operatorname{hofib}}
\newcommand{\Ext}{{\operatorname{Ext}}}
\newcommand{\Spec}{{\operatorname{Spec}}}

\newcommand{\BSL}{\mathrm{B}\mathbf{SL}}

\newcommand{\pia}{\underline{\pi}^{\A^1}}
\newcommand{\sheafS}{\mathbf{S}_4}
\newcommand{\Kalg}{\mathbf{K}^{\mathsf{Q}}}

\newcommand{\PP}{\mathbf{P}}

\newcommand{\kq}{\mathbf{kq}}

\newcommand{\D}{\mathsf{D}}
\newcommand{\E}{\mathsf{E}}

\newcommand{\F}{\mathsf{F}}
\newcommand{\GG}{\mathsf{G}}

\newcommand{\f}{\mathsf{f}}

\newcommand{\GW}{\mathbf{GW}}

\newcommand{\Mil}{\mathsf{M}}
\newcommand{\MilWitt}{\mathsf{MW}}

\newcommand{\KMil}{\mathbf{K}^{\Mil}}

\newcommand{\kmil}{\mathbf{k}^{\Mil}}

\newcommand{\KMW}{\mathbf{K}^{\MilWitt}}

\newcommand{\ZZ}{{\mathbb Z}}

\newcommand{\CC}{\mathbb{C}}
\newcommand{\sphere}{\mathbb{S}}

\newcommand{\pr}{{\mathrm{pr}}}
\newcommand{\id}{{\mathrm{id}}}

\newcommand{\cone}{\operatorname{cone}}

\newcommand{\A}{\mathbf{A}}
\newcommand{\G}{\mathbf{G}_{\mathfrak{m}}}

\newcommand{\CP}{\mathbf{CP}}
\newcommand{\RP}{\mathbf{RP}}

\newcommand{\MZ}{\mathbf{M}\mathbb{Z}}

\newcommand{\GL}{\mathbf{GL}}
\newcommand{\SL}{\mathbf{SL}}

\newcommand{\SU}{\mathbf{SU}}

\newcommand{\Top}{\mathsf{top}}

\newcommand{\Sm}{\mathbf{Sm}}

\newcommand{\Hom}{\operatorname{Hom}}

\title{Endomorphisms of the projective plane and the image of the Suslin-Hurewicz map}
\author{Oliver R\"ondigs}
\date{October 12, 2021, revised on July 9 and December 2, 2022}
\begin{document}
\maketitle
\begin{abstract}
  The endomorphism ring of the projective plane over a field
  $F$ of characteristic neither two nor three is slightly more complicated in the
  Morel-Voevodsky motivic stable
  homotopy category than in Voevodsky's derived category
  of motives. In particular, it is not commutative precisely if
  there exists a square in $F$ which does not admit a sixth root.
  A byproduct of these computations is a proof of Suslin's
  conjecture on the Suslin-Hurewicz homomorphism
  from Quillen to Milnor $K$-theory in degree four,
  based on work of Asok, Fasel, and Williams \cite{afw}.
\end{abstract}

\section{Introduction}
\label{section:introduction}

Automorphisms of geometric objects describe their symmetries,
and hence important geometric information. It depends on the
context which type of morphisms are considered useful.
In the case of a projective space $\PP^n$ over the complex numbers, one
may consider its linear automorphisms (a group denoted $\mathrm{PGL}_{n+1}(\CC)$),
birational automorphisms (the Cremona group $\mathrm{Cr}_n(\CC)$),
diffeomorphisms, homeomorphisms, and self-homotopy-equivalences,
just to name a few. The Morel-Voevodsky $\A^1$-homotopy theory
provides an interesting way to consider self-homotopy-equivalences of
varieties \cite{mv}. Although this setup is conceptionally very
satisfying,
concrete determinations of endomorphisms in the $\A^1$-homotopy
category are hard to come by. For example, the
endomorphism ring of $\PP^1$ over a perfect field $F$ in the
pointed $\A^1$-homotopy category is given by the Grothendieck ring
of isomorphism classes of symmetric inner product spaces over $F$ with
a chosen basis, where the isomorphisms preserve the inner product and
have determinant 1 with respect to the chosen bases
\cite[Remark 7.37]{morel.at}.

Stabilization with respect to smashing
with a projective line $\PP^1\smash -$ provides a simpler
categorical setting, the motivic stable homotopy category $\SH(F)$, which
is still richer than the corresponding derived category of motives
\cite{voevodsky.icm}. Part of the gain from leaving the unstable realm
is an additive (in fact triangulated)
structure, whence the set of endomorphisms of any object is always
a ring. For example, it is a deep theorem of Morel's that the
endomorphism ring of the projective line in the motivic stable
homotopy category over a field $F$ is the Grothendieck-Witt ring
of symmetric bilinear forms with coefficients in $F$ \cite{morel.zeroline}.
The
addition in the Grothendieck-Witt ring, whose elements are formal differences
of symmetric bilinear forms, is induced by orthogonal sum, and the
multiplication by tensor product of forms. 
By construction, it coincides with the endomorphism ring of the
unit for the symmetric monoidal structure given by the smash product.
Hence it has to be commutative. This is already different for the projective plane.

\begin{theorem*}
  Let $F$ be a field of characteristic neither 2 nor 3, with group of units $F^\units$.
  The endomorphism ring $[\PP^2,\PP^2]_{\SH(F)}$ in the motivic
  stable homotopy category of $F$ has an underlying additive
  group isomorphic to $\ZZ\directsum\ZZ\directsum F^\units/(F^\units)^6$. The
  multiplica\-tion corresponds to the multiplication given by
  \[   (x_1,x_2,x_3)\circ (y_1,y_2,y_3) =  (x_1y_1,x_1y_2+x_2y_1+2x_2y_2,x_1y_3+x_3y_1+2x_3y_2). \] 
\end{theorem*}

In particular, the ring $[\PP^2,\PP^2]_{\SH(F)}$ is non-commutative if and only if
there exists a square in $F$ which does not admit a sixth root, or, equivalently,
if the cube map $u\mapsto u^3$ is not surjective on $F$.
Its group of units (which could
be called the group of $\PP^1$-stable self-$\A^1$-homotopy-equivalences
of $\PP^2$) consists of all triples $(x_1,x_2,x_3)\in \ZZ\directsum\ZZ\directsum
F^\units/(F^\units)^6$ where either $x_1=\pm 1$ and $x_2=0$, or
$x_1=\pm 1 $ and $x_2=-x_1$. It is as non-commutative as the endomorphism ring
it belongs to. Along the way, the homotopy modules $\pi_1\PP^2$ and
$\pi_2\PP^2$ will be determined, based on computations in \cite{rso.oneline}
and \cite{rondigs.moore}. These computations
provide an ingredient to complete the
program Aravind Asok, Jean Fasel and Ben Williams
developed in \cite{afw} to prove Suslin's
conjecture on the Suslin-Hurewicz homomorphism
from Quillen to Milnor $K$-theory in degree four.

\begin{theorem*}
  Let $F$ be an infinite field of characteristic different from $2$ and $3$,
  and $A$ an essentially smooth local $F$-algebra.
  The image of the Suslin-Hurewicz homomorphism
  $K^{\mathrm{Quillen}}_4(A) \to K^{\mathrm{Milnor}}_4(A)$ coincides with
  $6K^{\mathrm{Milnor}}_4(A)$.
\end{theorem*}

This closes the gap between three and five in the set of
degrees for which Suslin's conjecture was
previously known. See the introduction of \cite{afw} for
more details on its history, as well as Suslin's original paper \cite{suslin.1046}.

\section{Topology}
\label{sec:topology}

Let $\CP^n$ denote complex projective space of complex dimension $n$.
This section contains rather elementary calculations in the classical
stable homotopy theory, which determine the endomorphism ring of
$\CP^2$ in the stable homotopy category.
These calculations are based on stable homotopy groups of spheres
$\pi_m\sphere$
in degree $m<6$ and the action of the topological Hopf map
$ \eta\colon S^3 \to \CP^1\iso S^2$, whose cofiber is $\CP^2$, on them.
The standard reference here is \cite{toda}. As is customary in stable
homotopy theory, the notation for a map and its (de)suspensions
coincide if the context allows it.
The purpose of this section
is not to present original results (there aren't any), but instead
to document the similarities and differences to the
situation in the motivic stable homotopy category.

Choose a basepoint for $\CP^1$, and hence $\CP^2$,
which will not appear in the notation.
The cofiber sequence
\begin{equation}\label{eq:cof-cp2}
  S^3\xrightarrow{\eta} \CP^1 \xrightarrow{i}\CP^2\xrightarrow{q} S^4
\end{equation}
induces a long exact sequence of stable homotopy groups
\[ \dotsm \xrightarrow{\eta} \pi_m\CP^1 = \pi_{m-2}\sphere \xrightarrow{i_\ast}
  \pi_m\CP^2 \xrightarrow{q_\ast} \pi_m S^4=\pi_{m-4}\sphere \xrightarrow{\eta}
  \pi_{m-1}\CP^1=\pi_{m-3}\sphere \to \dotsm \]
terminating with $\pi_2\CP^1=\pi_2\CP^2$. The induced short
exact sequences
\[ 0\to \pi_{m-2}\sphere/\eta\pi_{m-3}\sphere \to \pi_m\CP^2 \to {}_\eta\pi_{m-4}\sphere \to 0\]
express the stable homotopy group of the complex projective plane
as an extension of two groups,  the subgroup annihilated by $\eta$, and
the cokernel of
multiplication by $\eta$, on the respective stable homotopy group of spheres.
Since $\eta\colon \pi_0\sphere\iso \ZZ \to \pi_1\sphere\iso \ZZ/2\ZZ$ is surjective,
$\eta\colon \pi_1\sphere\to \pi_2\sphere$ is an isomorphism,
$\eta\colon \pi_2\sphere \to \pi_3\sphere \iso \ZZ/24$ is injective, and
$\pi_4\sphere\iso \pi_5\sphere\iso 0$, the following table results, without
any extension problem to solve.
\begin{align*}
  m  & &2 & & 3 && 4 && 5 && 6 && 7\\
  \pi_m\CP^2 & &\ZZ && 0 && 2\ZZ && \ZZ/12 && 0 && \ZZ/24 
\end{align*}
Here $2\ZZ$ denotes the abelian group of even integers under addition.
The vanishing $\pi_{3}\CP^2=0$ implies that
the cofiber sequence~(\ref{eq:cof-cp2}) induces
a short exact sequence
\begin{equation}\label{eq:endo-cp2}
  0 \to \pi_4\CP^2=[S^4,\CP^2] \xrightarrow{q^\ast}
  [\CP^2,\CP^2]\xrightarrow{i^\ast}
  [\CP^1,\CP^2]=\pi_2\CP^2\to 0
\end{equation}
of stable homotopy groups.
Hence the abelian group $[\CP^2,\CP^2]$ is an extension of
two free abelian groups, each
on one generator. Since $\pi_2\CP^2\iso \ZZ$, the short exact
sequence~(\ref{eq:endo-cp2}) splits. In order to describe the
ring structure, it helps to be more specific. A generator for
$\pi_2\CP^2$ is the inclusion $i\colon \CP^1\hookrightarrow \CP^2$.
It is the image of $\id_{\CP^2}$ under
$[\CP^2,\CP^2]\xrightarrow{i^\ast} [\CP^1,\CP^2]$.
To describe a generator for $\pi_4\CP^2$, observe that
there exists a unique map
$\omega\colon S^4\to \CP^2$ such that $q\circ \omega = 2\id_{S^4}$.
The short exact sequence~(\ref{eq:endo-cp2}) then implies that
every element $x\in [\CP^2,\CP^2]$ can uniquely be expressed
as a sum $x_1\id_{\CP^2}+x_2(\omega\circ q)$, where
$x_1,x_2\in \pi_0\mathbb{S}\iso \ZZ$.
The ring structure is then given as
\begin{align*}
  x\circ y & = \bigl(x_1\id_{\CP^2}+x_2(\omega\circ q)\bigr)\circ \bigl(y_1\id_{\CP^2}+y_2(\omega\circ q)\bigr) \\
           & = x_1y_1\id_{\CP^2} +(x_1y_2+x_2y_1)(\omega\circ q) + x_2y_2(\omega\circ q\circ \omega\circ q)\\
           & = x_1y_1\id_{\CP^2} +(x_1y_2+x_2y_1+2x_2y_2)(\omega\circ q) 
\end{align*}
and is in particular commutative. The group of units consists of the following
4 elements:
$\{\id_{\CP^2},-\id_{\CP^2},\id_{\CP^2}-\omega\circ q,-\id_{\CP^2}+\omega\circ q\}$
Of course knowledge does not stop at the dimension two. For example,
\cite{mukai.stable-complex} determines the groups $[\CP^n,\CP^n]$ for
$n\leq 7$. The complex dimension 7 is the smallest dimension where this group
contains torsion; in fact, $[\CP^7,\CP^7]\iso \ZZ^7\directsum \ZZ/2\ZZ$.
The ring structure is commutative in all these dimensions.

The table above allows to determine $[\Sigma\CP^2,\CP^2]$ as well,
which is useful, because this group contains the interesting element
$\eta\id_{\CP^2}=\eta\smash \CP^2$.
The suspension of the homotopy cofiber sequence~(\ref{eq:cof-cp2})
induces a long exact sequence on $[-,\CP^2]$.
Since $\pi_3\CP^2$ is the zero group, there results an isomorphism
$\pi_{5}\CP^2/\eta\pi_4\CP^2 \xrightarrow{\iso} [\Sigma\CP^2,\CP^2].$
As $\pi_5\CP^2$ is cyclic, generated by $i\circ \nu$, the abelian group
$[\Sigma\CP^2,\CP^2]$ is generated by the map
\[ \Sigma\CP^2\xrightarrow{q} S^5 \xrightarrow{\nu} S^2 = \CP^1 \xrightarrow{i} \CP^2.\]
Its order can be determined by identifying $\omega\circ \eta\in \pi_5\CP^2$,
which is $\pm 6(i\circ \nu)$, as the Toda bracket
$\langle \eta,2,\eta\rangle = \{6\nu,-6\nu\}$ shows, together with
\cite[Prop.~1.8]{toda} applied to
$\alpha=\gamma=\eta$ and $\beta=2$, using $\omega=\widetilde{\beta}$.
(See Appendix~\ref{sec:toda-brackets} for a definition and some properties of Toda brackets, and in particular Proposition~\ref{prop:toda-brackets-coext}
for a restatement of \cite[Prop.~1.8]{toda}.)
Hence
$[\Sigma\CP^2,\CP^2]$ is cyclic of order $6$. The element $\eta\id_{\CP^2}$
turns out to be the unique nonzero element of order 2, as the Toda
bracket $\langle \eta,2=q\circ \omega,\eta\rangle$ also implies. The properties
of Toda brackets supply an inclusion
\[ \langle \eta,q,\eta\smash \CP^2\rangle\circ \omega \subset
\langle \eta,q,\eta\smash \CP^2\circ \omega =\omega\circ \eta \rangle
\subset \langle \eta,q\circ \omega,\eta\rangle = \{6\nu,-6\nu\} \]
which shows that it does not contain the zero element.
More precisely, since the
composition
\[ \pi_3\sphere \xrightarrow{q^\ast} [ \Sigma^2\CP^2,S^3 ]
\xrightarrow{\omega^\ast} \pi_3\sphere \]
is multiplication with $2$ on a cyclic group with 24 elements, and
the homomorphism $q^\ast\colon \pi_3\sphere \to [\Sigma^2\CP^2,S^3]$ is the
projection onto a cyclic group with 12 elements (as one deduces from
the action of $\eta$ on $\pi_n\sphere$ for $n\in \{1,2\}$), the homomorphism
$\omega^\ast\colon [\Sigma^2\CP^2,S^3]\to \pi_3\sphere$ is injective.
One obtains
$\langle \eta,q,\eta\smash \CP^2\rangle \subset \{3\nu\circ q,-3\nu\circ q\}$.
Hence the identity $\id_{\CP^2}\in [\CP^2,\CP^2]$ satisfies
\begin{equation}\label{eq:etacp2}
  \eta\id_{\CP^2}=\eta\smash \CP^2 = \pm 3 (i \circ \nu\circ q).
\end{equation}

For comparison purposes with the motivic situation,
it is instructive to look at the real case as well.
Let $\RP^n$ denote real projective space.
The cofiber
sequence\footnote{It
  would be better to use $-2$ and not $2$, since the real realization of the algebraic Hopf map is
  negative.}
\begin{equation}\label{eq:cof-rp2}
  S^1\xrightarrow{2} \RP^1 \xrightarrow{i}\RP^2\xrightarrow{q} S^2
\end{equation}
induces a long exact sequence of stable homotopy groups
\[ \dotsm \xrightarrow{2} \pi_m\RP^1 = \pi_{m-1}\sphere \xrightarrow{i_\ast}
  \pi_m\RP^2 \xrightarrow{q_\ast} \pi_m S^2=\pi_{m-2}\sphere \xrightarrow{2}
  \pi_{m-1}\RP^1=\pi_{m-2}\sphere \to \dotsm \]
terminating with $\pi_1\RP^1\to \pi_1\RP^2$. The induced short
exact sequences
\[ 0\to \pi_{m-1}\sphere/2\pi_{m-1}\sphere \to \pi_m\RP^2 \to {}_2\pi_{m-2}\sphere \to 0\]
express the stable homotopy group of the real projective plane
as an extension of two groups,  the $2$-torsion subgroup, and
the cokernel of
multiplication by $2$, on the respective stable homotopy group of spheres.
The groups $\pi_3\RP^2$ and $\pi_4\RP^2$ are both extensions of
$\ZZ/2$ by $\ZZ/2$. The Toda bracket $\langle 2, \eta, 2\rangle = \{\eta^2\}$
implies that $\pi_3\RP^2$ is given by the nontrivial extension,
the extension for $\pi_4\RP^2$ turns out to be trivial \cite[Lemma 5.2]{wu.p2}.
The following table results.
\begin{align*}
  m  & &1 & & 2 & & 3 && 4 && 5 && 6 \\
  \pi_m\RP^2 & &\ZZ/2\ZZ && \ZZ/2\ZZ && \ZZ/4\ZZ && \ZZ/2\ZZ\times\ZZ/2\ZZ && \ZZ/2\ZZ && 0 
\end{align*}
The portion for $m<3$ of this table implies that
the cofiber sequence~(\ref{eq:cof-rp2}) induces
a short exact sequence
\begin{equation}\label{eq:endo-rp2}
  0 \to \pi_2\RP^2=[S^2,\RP^2] \xrightarrow{q^\ast}
  [\RP^2,\RP^2]\xrightarrow{i^\ast}
  [\RP^1,\RP^2]=\pi_1\RP^2\to 0
\end{equation}
and hence the abelian group $[\RP^2,\RP^2]$ is an extension of
two groups of order two. As in the computation of $\pi_3\RP^2$,
the extension is nontrivial, meaning that $\id_{\RP^2}\in [\RP^2,\RP^2]$
is an element of order 4, as already proven in \cite{barratt.track2}.
There cannot be any doubt whatsoever regarding the ring structure of
$[\RP^2,\RP^2]$.

\section{Over a field}
\label{sec:over-field}

Let $F$ be a field, and let $\SH(F)$ denote the motivic stable homotopy
category of $F$ \cite{voevodsky.icm}. For a motivic spectrum $\E\in \SH(F)$
and integers $s,w\in \ZZ$,
let $\pi_{s,w} \E$ denote the abelian group 
$[\Sigma^{s,w}\unit,\E]$, where $\E$ is a motivic spectrum and
$\unit_F=\unit$ is the motivic sphere spectrum. The grading conventions
are such that the suspension functor $\Sigma^{2,1}=\Sigma^{1+(1)}$
is suspension with $\PP^1$, and $\Sigma^{1,0}=\Sigma^{1+(0)}=\Sigma^1=\Sigma$
is suspension with the simplicial circle. Set $\pi_{s+(w)}\E:=\pi_{s+w,w}\E$, 
and let
\[ \pi_{s+\bideg}\E = \bigoplus_{w\in \ZZ} \pi_{s+w,w}\E = \bigoplus_{w\in \ZZ} \pi_{s+(w)}\E \]
denote the direct sum, considered as a $\ZZ$-graded module over the
$\ZZ$-graded ring $\pi_{0+\bideg}\unit$. The notation
$\pi_{s-\bideg}\E:=\pi_{s+(-\star)}\E$ will be used frequently.
The strictly 
$\A^1$-invariant sheaf obtained as the associated Nisnevich sheaf
of $U\mapsto [\Sigma^{s,w}U_+,\E]$ for $U\in \Sm_F$ is denoted
$\underline{\pi}_{s,w}\E$, which gives rise to
the homotopy module $\underline{\pi}_{s+\bideg}\E$.
In the following, every occurrence of ``$\pi$'' can be replaced
by ``$\underline{\pi}$'' without affecting the truth of the (suitably
reinterpreted) statements.
See \cite{morel.zeroline} for the following fundamental result.

\begin{theorem}[Morel]\label{thm:zeroline}
  Let $F$ be a field. Then $\pi_{0-\bideg}\unit$ is the Milnor-Witt
  $K$-theory of $F$.
\end{theorem}

The Milnor-Witt $K$-theory of $F$ is denoted $\KMW(F)$, or
simply $\KMW$, following the convention that the base field
or scheme may be ignored in the notation. The definition and
some details regarding
$\KMW$ and modules over it are contained in the
Appendix~\ref{sec:modules-over-milnor}.
Theorem~\ref{thm:zeroline} implies that for every motivic spectrum
$\E$ and for every integer $s$, $\pi_{s+\bideg}\E$ has a canonical
structure as a graded
$\KMW$-module. The conventions dictate that this structure comes with
a sign change in the sense that
elements in $\KMW_d$ provide homomorphisms
$\pi_{s+(w)}\E\to \pi_{s+(w-d)}\E$ for every $d,s,w\in \ZZ$.

Choose a basepoint for $\PP^1$, and hence $\PP^2$,
which will not appear in the notation. Neither will the base field $F$ most
of the time.
The main cofiber sequence over a field is 
\begin{equation}\label{eq:cof-p2}
  S^{1+(2)}\xrightarrow{\eta} \PP^1 \xrightarrow{i}\PP^2\xrightarrow{q} S^{2+(2)}.
\end{equation}
It induces a long exact sequence of $\KMW$-modules
\[ \dotsm \xrightarrow{\eta} \pi_{m+\bideg}\PP^1 
  \xrightarrow{i_\ast}
  \pi_{m+\bideg}\PP^2 \xrightarrow{q_\ast} \pi_{m+\bideg} S^{2+(2)}
  \xrightarrow{\eta}
  \pi_{m-1+\bideg}\PP^1
  \to \dotsm \]
terminating with $\pi_{1+\bideg}\PP^2$ by connectivity \cite{morel.connectivity}.
The induced short exact sequences
\begin{equation}\label{eq:ses-p2}
  0\to \pi_{m-1+(\star-1)}\unit/\eta\pi_{m-1+(\star-2)}\unit \to \pi_{m+\bideg}\PP^2 \to {}_\eta\pi_{m-2+(\star-2)}\unit \to 0
\end{equation}
express $\pi_{m+\bideg}\PP^2$ 
as an extension of two $\KMW$-modules, the submodule of $\pi_{m-2+(\star-2)}\unit$
annihilated by $\eta$,
and the cokernel of
multiplication by $\eta$ on $\pi_{m-1+(\star-1)}\unit$.
This justifies the relevance
of the following statement. 

\begin{theorem}[Gille-Scully-Zhong]\label{thm:ker-eta-pi0}
  Let $F$ be a field of characteristic not two.
  The submodule of $\KMW$ annihilated by $\eta$ coincides with the image of 
  multiplication by the hyperbolic plane on $\KMW$:
  \[ {}_\eta\KMW = \hyper\KMW \]
\end{theorem}

\begin{proof}
  This follows from the injectivity of the
  homomorphism $\epsilon_\star^R$ in \cite[Theorem 5.4]{gsz}
  for $R$ a field.
\end{proof}

Theorem~\ref{thm:ker-eta-pi0}
provides the exactness of the sequence mentioned in
\cite[Remark 4.3]{choudhury-hogadi}. It is quite special that the
kernel of multiplication by $\eta$ on $\KMW$ is generated by
a single element, but then the element $\eta$ is also quite special.
As a consequence of Theorem~\ref{thm:ker-eta-pi0}, the
short exact sequence
\[ 0\to \pi_{1+(\star-1)}\unit/\eta\pi_{1+(\star-2)}\unit \to \pi_{2+\bideg}\PP^2 \to {}_\eta\pi_{0+(\star-2)}\unit \to 0\]
specializes to an isomorphism
$\pi_{1+(w-1)}\unit/\eta\pi_{1+(w-2)}\unit\iso \pi_{2+(w)}\PP^2$
for $w>2$. Therefore knowing $\pi_{1+(\star-1)}\unit/\eta\pi_{1+(\star-2)}\unit$
is essential. 

\begin{theorem}\label{thm:pi1modeta}
  Let $F$ be a field of characteristic not two or three.
  The unit map $\unit\to \kq$ induces a surjection
  $\pi_{1+\bideg}\unit/\eta\pi_{1+(\star-1)}\unit\to \pi_{1+\bideg}\kq/\eta\pi_{1+(\star-1)}\kq$
  whose kernel is $\KMil_{2-\star}/12$ (generated in $\star=2$)
  after inverting the exponential characteristic $e$ of $F$.
  In particular, the vanishing $\pi_{1+(w)}\kq/\eta\pi_{1+(w-1)}\kq$ for $w>1$
  implies that the nontrivial group of highest weight
  is $\pi_{1+(2)}\unit[e^{-1}]/\eta\pi_{1+(1)}\unit[e^{-1}]\iso \ZZ/12$,
  generated by the image of the Hopf map $\nu$.
\end{theorem}

\begin{proof}
  This is a consequence of \cite[Theorem 5.5]{rso.oneline} in the formulation
  given in \cite[Theorem 2.5]{rondigs.moore}; see
  also \cite[Theorem 1.1]{rso.twoline}.
  First of all, the unit map
  $\pi_{1+\bideg}\unit\to \pi_{1+\bideg}\kq$ is surjective, whence the same
  is true for the induced map on the quotients. Set $e$ to be the exponential
  characteristic of $F$.
  Consider the following natural transformation
  \begin{center}
    \begin{tikzcd}
      0 \ar[r]  & \KMil_{2-\star}/24[\tfrac{1}{e}] \ar[r] \ar[d] &
      \pi_{1+\bideg}\unit[\tfrac{1}{e}]   \ar[r] \ar[d] &
      \pi_{1+\bideg}\kq[\tfrac{1}{e}] \ar[r] \ar[d] & 0 \\
      0 \ar[r]  & A \ar[r]  & \pi_{1+\bideg}\unit[\tfrac{1}{e}]/\eta
      \ar[r]  &  \pi_{1+\bideg}\kq[\tfrac{1}{e}]/\eta \ar[r] & 0
    \end{tikzcd}
  \end{center}
  of short exact sequences. 
  The snake lemma implies that 
  the map $\KMil/24[\tfrac{1}{e}]\to A$ is surjective,
  because the map $\eta\pi_{1+\bideg}\unit \to \eta\pi_{1+\bideg}\kq$
  on the kernels is surjective. The map $\eta\colon \Sigma^{(1)}\unit\to \unit$
  factors by construction as $\eta\colon \Sigma^{(1)}\unit\to \f_1\unit\to \unit$,
  where $\f_1\unit \to \unit$ denotes the first effective cover. The presentations
  given in \cite[Lemma 2.3, Theorem 2.5]{rondigs.moore} then imply that
  the kernel of the map $\eta\pi_{1+\bideg}\unit \to \eta\pi_{1+\bideg}\kq$
  is generated by $12\nu=\eta^2\eta_\Top$, where $\eta_\Top\in \pi_{1+(0)}\unit$
  denotes the topological Hopf map.\footnote{This element appeared in Section~\ref{sec:topology} without the subscript.}
  Hence the snake lemma also implies that $A\iso \KMil/12[\tfrac{1}{e}]$.
\end{proof}

The same proof shows that Theorem~\ref{thm:pi1modeta} is valid in
characteristic 3 as well in the sense
that the kernel of
\[ \pi_{1+\bideg}\unit/\eta\pi_{1+(\star-1)}\unit\to \pi_{1+\bideg}\kq/\eta\pi_{1+(\star -1)}\kq\]
is isomorphic to $\KMil/4[\tfrac{1}{3}]$ after inverting $3$.
Theorems~\ref{thm:ker-eta-pi0} and~\ref{thm:pi1modeta} provide
sufficient information about the outer terms in the short exact sequence
\begin{equation}\label{eq:pi2p2}
  0 \to \pi_{1+(\star-1)}\unit/\eta\pi_{1+(\star-2)}\unit\to \pi_{2+\bideg}\PP^2
  \to {}_{\eta}\pi_{0+(\star-2)}\unit \to 0
\end{equation}
of $\KMW$-modules. Actually the outer terms are $\KMil$-modules in
a natural way; $\eta$ acts trivially on these. However, $\eta$
acts nontrivially on the middle term. The reason is the
Toda bracket $\langle \eta,\hyper,\eta\rangle = \{6\nu,-6\nu\}$
from \cite[Proposition 4.1]{rondigs.moore}. 

\begin{lemma}\label{lem:ext-pi1p2}
  Let $F$ be a field of characteristic neither $2$ nor $3$.
  The action of $\eta$ on the $\KMW$-module $\pi_{2+\bideg}\PP^2$
  in the extension
  \[0 \to \pi_{1+(\star-1)}\unit/\eta\pi_{1+(\star-2)}\unit\to \pi_{2+\bideg}\PP^2
  \to {}_{\eta}\pi_{0+(\star-2)}\unit  \to 0\]
  is determined by the fact that $\liftofhyper\circ \eta=6 (i\circ \nu)$,
  where $\liftofhyper\in \pi_{2+(2)}\PP^2$ is any lift of
  $\hyper\in \pi_{0+(0)}\unit$.
\end{lemma}

\begin{proof}
  The Toda bracket $\langle \eta,\hyper,\eta\rangle = \{6\nu,-6\nu\}$
  from \cite[Proposition 4.1]{rondigs.moore} implies by 
  Proposition~\ref{prop:toda-brackets-coext}
  that there exists an element $\liftofhyper\in \pi_{2+(2)}\PP^2$ which on the
  one hand maps to $\hyper\in {}_{\eta}\pi_{0+(0)}\unit$,
  and on the other hand is such
  that $\liftofhyper\circ \eta$ is the image of $6\nu\in \pi_{1+(2)}\unit$.
  Inspecting the
  short exact sequence~(\ref{eq:pi2p2}) in weight $3$ gives
  an isomorphism $\pi_{1+(2)}\unit/\eta\pi_{1+(1)}\unit \iso \pi_{2+(3)}\PP^2$
  by Theorem~\ref{thm:ker-eta-pi0}.
  Hence $\pi_{2+(3)}\PP^2$ is cyclic of order 12 by Theorem~\ref{thm:pi1modeta},
  with the image $i\circ \nu$ of
  $\nu$ as a generator. It follows that $\liftofhyper\circ \eta$ is the
  unique nonzero element
  of order two in this group, and this is true for
  any choice of $\liftofhyper$ lifting $\hyper$.
  Inspecting the
  short exact sequence~(\ref{eq:pi2p2}) in weight $2$ provides
  \[ 0 \to \pi_{1+(1)}\unit/\eta\pi_{1+(0)}\unit \to \pi_{2+(2)}\PP^2
  \to {}_{\eta}\pi_{0+(0)}\unit \to 0.\]
  that for any two lifts $\liftofhyper,\hyper^{\prime\prime}$ of $\hyper$,
  there exists $\{u\}\in \KMil_1/12$
  with $\liftofhyper-\hyper^{\prime\prime}=i\{u\}\nu$.

  In order to describe the extension group more precisely,
  set $A_\star:= \pi_{1-(\star-1)}\unit/\eta\pi_{1-(\star)}\unit$. The
  extension~(\ref{eq:pi2p2}) is given by an element in
  \begin{equation*}
    \Ext^1_{\KMW}({}_{\eta}\pi_{0-(\star+2)}\unit,A_{\star+2})  \iso
    \Ext^1_{\KMW}(2\KMil_{\star+2},A_{\star+2})  \iso
    \Ext^1_{\KMW}(2\KMil,A)
  \end{equation*}
  by Theorem~\ref{thm:ker-eta-pi0}.
  The short exact sequence
  \[ 0\to {}_{2}\KMil\to \KMil \to 2\KMil \to 0 \]
  of $\KMW$-modules induces a long exact sequence
  \[ 0\to \Hom(2\KMil,A)\to \Hom(\KMil,A) \to \Hom({}_{2}\KMil,A) \to
  \Ext^1(2\KMil,A)\to \Ext^1(\KMil,A) \to \dotsm \]
  where the subscript ``$\KMW$'' is suppressed.
  Lemma~\ref{lem:extkm} applies to provide an isomorphism
  between the group
  $\Ext^1_{\KMW}(\KMil,A=\pi_{1-(\star -1)}\unit/\eta\pi_{1-(\star)}\unit)$
  and 
  \[ {}_{\hyper}(\pi_{1+(2)}\unit/\eta\pi_{1+(1)}\unit) \iso {}_{2}(\pi_{1+(2)}\unit/\eta\pi_{1+(1)}\unit) \iso {}_{2}\ZZ/12 \]
  where the last isomorphism follows from Theorem~\ref{thm:pi1modeta}.
  Note that multiplication with $\eta$ on the $\KMW$-module
  $\pi_{1+\bideg}\unit/\eta\pi_{1+\bideg}\unit$ is the zero homomorphism by
  construction, which simplifies the term appearing in Lemma~\ref{lem:extkm}.
  Hence $\Ext^1_{\KMW}(\KMil,\pi_{1-(\star -1)}\unit/\eta\pi_{1-(\star)}\unit)
  =\{0,6\nu\}$
  does not depend on the base field $F$. The homomorphism
  $\Ext^1(2\KMil,A)\to \Ext^1(\KMil,A)$ is surjective,
  because the extension in question corresponds to the unique
  nonzero element in $\A_{-1}=\pi_{1+(2)}\unit/\eta\pi_{1+(1)}\unit$ of order two
  by the Toda bracket $\langle \eta,\hyper,\eta\rangle = \{6\nu,-6\nu\}$
  from \cite[Proposition 4.1]{rondigs.moore}.
  There results an exact sequence
  \begin{equation}\label{eq:ext-2KMil}
    \Hom_{\KMW}(\KMil,A)\to \Hom_{\KMW}({}_{2}\KMil,A)\to \Ext_{\KMW}(2\KMil,A)\to \Ext_{\KMW}(\KMil,A)\to 0
  \end{equation}
  where $\Hom({}_{2}\KMil,A)\subset \Hom(\KMil(-1),A)=A_1$ via the
  surjection $\KMil(-1)\to {}_{2}\KMil$ obtained by multiplying with
  $\{-1\}\in \KMil_1$, see Theorem~\ref{thm:ker-2-kmil}.
  Since $\Hom_{\KMW}(\KMil,A)=\Hom_{\KMil}(\KMil,A)$ and
  $\Hom_{\KMW}({}_{2}\KMil,A)=\Hom_{\KMil}({}_{2}\KMil,A)$, the exact sequence~(\ref{eq:ext-2KMil})
  induces a short exact sequence
  \[ 0 \to \Ext^1_{\KMil}(2\KMil,A)\to \Ext^1_{\KMW}(2\KMil,A)\to \Ext_{\KMW}(\KMil,A)\to 0\]
  in which $\Ext^1_{\KMil}(2\KMil,A)\iso \Hom_{\KMil}({}_{2}\KMil,A)/\rho A_0$.
  In particular, the sought-after
  element in $\Ext^1_{\KMW}(2\KMil,A)$ classifying
  the extension in question is determined by the relation
  $\liftofhyper\circ \eta = 6(i \circ \nu)$ and an element in the group
  $\Ext^1_{\KMil}(2\KMil,A)$ depending solely on the $\KMil$-module
  structures of $2\KMil$ and $A$.
\end{proof}

\begin{remark}\label{rem:unstable}
  Regarding the unstable situation, the $\A^1$-fiber sequence
  \[ \A^1\minus \{0\}\to \A^3\minus \{0\} \to \PP^2 \]
  and the $\A^1$-discreteness of $\A^1\minus \{0\}$ provide an identification
  $\underline{\pi}_2^{\A^1}\PP^2\iso \underline{\pi}^{\A^1}_2(\A^3\minus \{0\})
  \iso \underline{\mathbf{K}}^{\MilWitt}_3$
  of (unstable) $\A^1$-homotopy sheaves, 
  by \cite[Theorem 1.23]{morel.at}. Let $\underline{\pi}_{2+(3)}^{\A^1}\PP^2$
  denote the threefold contraction of $\pia_2\PP^2$, which coincides
  with the Nisnevich sheaf associated with the presheaf $X\mapsto \Hom_{\mathbf{H}^{\A^1}_\bullet(X)}(\A_X^3\minus \{0\},\PP^2_X)$ \cite[p.~72, Theorem 6.13]{morel.at}.
  The generator of
  $\underline{\pi}_{2+(3)}^{\A^1}\PP^2 \iso \underline{\mathbf{K}}^{\MilWitt}_0$
  is thus
  the class of the canonical map $\A^3\minus \{0\}\to \PP^2$.
  Stabilization with respect to $\PP^1$ provides a homomorphism
  \[ \underline{\pi}_{2+(3)}^{\A^1}\PP^2 \to \underline{\pi}_{2+(3)} \PP^2 \]
  to the stable homotopy sheaf computed in Lemma~\ref{lem:ext-pi1p2}.
  It sends the generator to $m(i\circ \nu)$, where $m$ is an integer
  unique up to multiples of 12, and $i\circ \nu$ is the generator of the
  target. A comparison with the classical topological situation via
  complex or \'etale realization, which is possible since the target
  does not depend on the base field, shows that $m=\pm 2$,
  because the order of the canonical map $S^5\to \CP^2$ is 6 after
  one suspension \cite[Theorem 1.2]{mukai.transfer}.
  Note that \'etale realization sends $\PP^n$ to the
  profinite completion of its complex realization
  $\CP^n$ by \cite[Theorem 12.9]{artin-mazur};
  see also \cite[Theorem 8.4]{friedlander}. The same applies to
  the maps involved here.
\end{remark}

The computations provided by $\pi_{1-(\star-1)}\PP^2\iso \KMil$ and
Lemma~\ref{lem:ext-pi1p2} suffice to conclude the
following statement.

\begin{theorem}\label{thm:endo-p2}
  Let $F$ be a field of characteristic not in $\{2,3\}$. The
  cofiber sequence~(\ref{eq:cof-p2}) induces a
  short exact sequence
  \begin{equation}\label{eq:endo-p2}
    0 \to \pi_{2+(\star+2)}\PP^2/\eta\pi_{2+(\star+1)}\PP^2 \to
    [\Sigma^{\bideg}\PP^2,\PP^2]\to
    \pi_{1+(\star+1)}\PP^2 \to 0
  \end{equation}
  of $\KMW$-modules.
  In particular, after inverting the exponential characteristic,
  there is an isomorphism
  \[ \ZZ/6\iso \pi_{2+(3)}\PP^2/\eta\pi_{2+(2)}\PP^2 \iso [\Sigma^{(1)}\PP^2,\PP^2]\]
  with $i\circ \nu \circ q$ as a generator, and an isomorphism
  \[ \ZZ\directsum\ZZ\directsum \KMil_1/6(F) \iso [\PP^2,\PP^2]\]
  of abelian groups, with $\id_{\PP^2}$ and $\liftofhyper\circ q$ each
  generating one free summand.
  The equality
  $\eta\cdot\id_{\PP^2}=\eta\smash\id{_{\PP^2}}=3(i\circ\nu\circ q)$ determines the
  action of $\eta$ on $[\Sigma^{\bideg}\PP^2,\PP^2]$.
\end{theorem}  

\begin{proof}
  As before, the cofiber sequence~(\ref{eq:cof-p2}) induces a short exact sequence
  \begin{equation*}
    0 \to \pi_{2+(\star+2)}\PP^2/\eta\pi_{2+(\star+1)}\PP^2 \to
    [\Sigma^{\bideg}\PP^2,\PP^2]\to
    {}_{\eta}\pi_{1+(\star+1)}\PP^2 \to 0.
  \end{equation*}
  The short exact sequence for $\pi_{1+\bideg}\PP^2$ specializes to the
  identification $\pi_{1+\bideg}\PP^2\iso \KMil$ mentioned already above.
  Since $\eta$ acts as zero on this $\KMW$-module,
  the short exact sequence~(\ref{eq:endo-p2}) of $\KMW$-modules
  follows. The identity $\id_{\PP^2}$ hits the canonical
  generator $i\in \pi_{1+(1)}\PP^2$.
  The action of $\eta$ in the $\KMW$-module structure on
  $[\Sigma^{\bideg}\PP^2,\PP^2]$ is then determined by
  specifiying $\eta \id_{\PP^2}\in [\Sigma^{(1)}\PP^2,\PP^2]$.
  More precisely, as in the proof of Lemma~\ref{lem:ext-pi1p2}
  there results a short exact sequence
  \[ 0 \to \Ext^1_{\KMil}(2\KMil,A)
  \to \Ext^1_{\KMW}(2\KMil,A)
  \to \Ext^1_{\KMW}(\KMil,A) \to 0 \]
  of abelian groups, where
  $A_\star := \pi_{2-(\star-2)}\PP^2/\eta\pi_{2-(\star-1)}\PP^2$.
  Lemma~\ref{lem:extkm} identifies the last extension group
  as
  \[ \Ext^1_{\KMW}(\KMil,\pi_{2-(\star-2)}\PP^2/\eta\pi_{2-(\star-1)}\PP^2) \iso {}_{\hyper}(\pi_{2+(3)}\PP^2/\eta\pi_{2+(2)}\PP^2)\iso {}_{2}\ZZ/6.\]
  Here the description of $\pi_{2+\bideg}\PP^2$ as a
  $\KMW$-module from Lemma~\ref{lem:ext-pi1p2} supplies the
  last isomorphism in this sequence, as well as the
  first isomorphism mentioned in the statement of the theorem. Hence
  $\eta\id_{\PP^2} = m (i\circ \nu\circ q)$ for some $m\in \ZZ$ which is
  unique up to multiples of 6. The element $\eta\id_{\PP^2}$
  turns out to be the unique nonzero element of order 2, as the Toda
  bracket $\langle \eta,\hyper=q\circ \liftofhyper,\eta\rangle$ implies.
  The properties
  of Toda brackets supply an inclusion
  \[ \langle \eta,q,\eta\smash \PP^2\rangle\circ \liftofhyper \subset
    \langle \eta,q,\eta\smash \PP^2 \circ \liftofhyper = \liftofhyper\circ \eta \rangle
    \subset \langle \eta,q\circ \liftofhyper,\eta\rangle = \{6\nu,-6\nu\} \]
  which shows that it does not contain zero. This already suffices to
  conclude.
  More precisely, since the composition
  \[ \pi_{1+(2)}\unit \xrightarrow{q^\ast} [ \Sigma^{1+(1)}\PP^2,S^{2+(1)} ]
    \xrightarrow{(\liftofhyper)^\ast} \pi_{1+(2)}\unit \]
  is multiplication with $\hyper$, one obtains
  $\langle \eta,q,\eta\smash \PP^2\rangle \subset \{3\nu\circ q,-3\nu\circ q\}$.
  Note that the homomorphism
  $[ \Sigma^{1+(1)}\PP^2,S^{2+(1)} ] \xrightarrow{(\liftofhyper)^\ast} \pi_{1+(2)}\unit$
  identifies with the inclusion $\ZZ/12\hookrightarrow \ZZ/24$, as one may
  deduce from the exact sequence
  \[  \pi_{1+(2)}\unit \xrightarrow{q^\ast}[ \Sigma^{1+(1)}\PP^2,S^{2+(1)} ]
    \xrightarrow{i^\ast} \pi_{0+(1)}\unit \xrightarrow{\eta ^\ast}\pi_{0+(2)}\unit.\]
  Hence $\id_{\PP^2}\in [\PP^2,\PP^2]$ satisfies
  $\eta\smash \PP^2 = 3 (i \circ \nu\circ q)$. 
\end{proof}

Theorem~\ref{thm:endo-p2} implies that
every element $x\in [\PP^2,\PP^2]$ can be expressed uniquely
as a sum $x_1\id_{\PP^2}+x_2(\liftofhyper\circ q)+x_3(i\circ \nu\circ q)$, where
$x_1,x_2\in \ZZ$ and $x_3\in K^\Mil_1/6$. It would probably be more honest
to think of the integers $x_1,x_2$ as the ranks of virtual quadratic forms.
In particular, the hyperbolic form ``$\hyper$'' corresponds to the integer ``2''.
Using that the composition $q\circ i$ is the zero map,
the ring structure is then given as
\begin{align*}
  x\circ y & = \bigl(x_1\id_{\PP^2}+x_2(\liftofhyper\circ q)+x_3(i\circ\nu\circ q)\bigr)\circ \bigl(y_1\id_{\PP^2}+y_2(\liftofhyper\circ q)+y_3(i\circ\nu\circ q)\bigr) \\
           & = x_1y_1\id_{\PP^2} +(x_1y_2+x_2y_1+2x_2y_2)(\liftofhyper\circ q) + (x_1y_3+x_3y_1+2x_3y_2)(i\circ \nu\circ q) 
\end{align*}
and in particular is not commutative if $2K^\Mil_1(F)/6$ is nonzero.
The group of units in $[\PP^2,\PP^2]$ consists of the elements
\[ \{\pm \id_{\PP^2}+x_3(i\circ\nu\circ q),\pm \id_{\PP^2}\mp \liftofhyper\circ q+x_3(i\circ\nu\circ q)\colon x_3\in K^\Mil_1/6\}. \]
If $F\subset \CC$, then complex realization
$[\Sigma^{(1)}\PP^2,\PP^2]\to [\Sigma\CP^2,\CP^2]$ is an isomorphism,
but $[\PP^2,\PP^2]\to [\CP^2,\CP^2]$ is possibly only surjective, not injective.
Nevertheless, every map in $[\PP^2,\PP^2]$ such that
its complex realization is a unit in $[\CP^2,\CP^2]$ is already a unit in
$[\PP^2,\PP^2]$.

\begin{remark}\label{rem:unit-p2}
  A different motivic type of endomorphisms of the projective plane
  occurs in the motivic stable homotopy category $\SH(\PP^2_F)$
  for $\PP^2_F$, with unit $\unit_{\PP^2_F}=\varphi^\ast(\unit_F)$. Here 
  $\varphi\colon\PP^2_F\to \Spec(F)$ is the structure morphism.
  The choice of a rational point provides a splitting
  $\unit_F\to \varphi_\sharp(\unit_{\PP^2_F}) \to \unit_F$
  of the counit,
  whence $\varphi_\sharp(\unit_{\PP^2_F})\simeq \PP^2_F\vee \unit_F$.
  Then $\varphi_\sharp$ induces
  a ring homomorphism 
  \[ \pi_{0+\bideg} \unit_{\PP^2_F} \xrightarrow{\varphi_\sharp}
    [\Sigma^{\bideg}\PP^2_+,\PP^2_+] \iso 
    [\Sigma^{\bideg}\PP^2,\PP^2] \directsum \pi_{0+\bideg}\unit_F
    \directsum [\Sigma^{\bideg}\PP^2,\unit] \]
  where the source
  \[ \pi_{0+\bideg} \unit_{\PP^2_F} 
    \iso [\Sigma^{\bideg}\varphi_\sharp(\unit_{\PP^2_F}),\unit_F]
    \iso [\Sigma^{\bideg}\PP^2,\unit_F] \directsum \pi_{0+\bideg}\unit_F\]
  is a commutative ring by definition. Using \cite[Theorem 2.7]{rondigs.moore},
  the short exact sequence
  \[ 0 \to \pi_{2+(2)}\unit_F/\eta\pi_{2+(1)}\unit_F \iso \KMil_2(F)/2
    \to [\PP^2,\unit] \to {}_{\eta}\pi_{1+(1)}\unit\iso \KMil_1(F)/24 \to 0 \]
  implies that $\pi_{0+(0)}\unit_{\PP^2_F}$ is not isomorphic to the
  Grothendieck-Witt ring of $\PP^2_F$, which is
  isomorphic to $\KMW_0(F)\directsum \KMil_0(F)$ \cite{walter.gw}.
\end{remark}

\section{Suslin's conjecture}
\label{sec:suslins-conjecture}

An application of some of the computations performed in Section~\ref{sec:over-field}
is a proof of Suslin's conjecture on the Hurewicz homomorphism from
Quillen to Milnor $K$-theory in degree four, exploiting the beautiful work~\cite{afw}.
Unstable homotopy sheaves will occur, as already in Remark~\ref{rem:unstable}.
Let
\[Q_{2n-1}:=\{(a,b)\in \A^n\times \A^n\colon \sum\limits_{j=1}^n a_jb_j = 1 \} \]
be the smooth affine quadric hypersurface which is weakly equivalent
to $\A^n\minus \{0\}$ via projection to the first $n$ coordinates \cite[Example 2.12(3)]{dugger-isaksen.hopf}.
The quotient scheme $Q_{2n-1}/\G$ with respect to
the free action $\lambda\cdot (a,b):=(\lambda a,\lambda^{-1}b)$
is then weakly equivalent via projection to the first $n$ coordinates
to $\PP^{n-1}$.
Given $(a,b,u)\in Q_{2n-1}\times \A^1\minus \{0\}$, let
$[a,b,u]\in \GL_n$ denote the matrix whose entry at $(j,k)$ is
$\delta_{jk}+(u-1)a_jb_k $, where $\delta_{jk}$ is the Kronecker symbol.
For a unit $u\in \G$, let $\Delta(u)$ denote the diagonal matrix
whose entries are $(u,1,\dotsc,1)$.
The map
\begin{equation}\label{eq:mapsln}
  \psi_n\colon Q_{2n-1}\times (\A^1\minus \{0\})\to \SL_n,\, \psi_n(a,b,u):= \Delta(u^{-1})\cdot [a,b,u]
\end{equation}
is compatible with the given $\G$ action on the first factor
(and trivial actions on
the other factor and $\SL_n$) and sends
$Q_{2n-1}\times \{1\}\cup \bigl\lbrace\bigl((1,0,\dotsc,0),(1,0,\dotsc,0)\bigr)\bigr\rbrace \times (\A^1\minus \{0\})$ to the identity matrix, the canonical
basepoint in $\SL_n$. Let
$\psi_n\colon \Sigma^{(1)}Q_{2n-1}/\G\to \SL_n$ denote also the induced 
pointed map, a variant of the map (with the same notation) to $\GL_n$
constructed in \cite[Section 5]{williams.stiefel}. Its complex
realization is denoted $j_n$ in \cite[p.~180]{mukai.transfer}, and
$f_{\SU(n)}$ in the even more classical source \cite[Section 4]{yokota}.
It is straightforward to check that the diagram
\begin{equation}\label{eq:diagram-psi}
  \begin{tikzcd}
    \Sigma^{(1)}Q_{2n-1}/\G \ar[d] \ar[r,"\psi_n"] & \SL_n \ar[d] \\
    \Sigma^{(1)}Q_{2n+1}/\G  \ar[r,"\psi_{n+1}"] & \SL_{n+1} 
  \end{tikzcd}
\end{equation}
with obvious inclusions as vertical maps commutes.

\begin{lemma}\label{lem:gmap2-weq}
  The map $\psi_2\colon \Sigma^{(1)}Q_3/\G \to \SL_2$
  is a weak equivalence over $\Spec(\ZZ)$.
\end{lemma}

\begin{proof}
  Let $\{b_1= 0\}\hookrightarrow Q_3$ 
  denote the smooth closed subscheme where $b_1 = 0$, and
  $\{b_1\neq 0\}\hookrightarrow Q_3$ 
  its open complement. The map
  \[\A^1\times (\A^1\minus \{0\})\to \{b_1=0\},\,
  (x,y)\mapsto (x,y,0,y^{-1})\]
  is an isomorphism.
  Its image $\overline{\{b_1=0\}}\hookrightarrow Q_3/\G$ 
  is a smooth closed subscheme isomorphic to $\A^1$, 
  with trivial normal bundle. The map
  $\psi_2\colon Q_3\times \A^1\minus \{0\} \to \SL_2$ sends
  the product $\{b_1=0\}\times \A^1\minus \{0\}$
  to the smooth closed subscheme $\{c_{21}=0\}\hookrightarrow \SL_2$,
  and the induced map
  $\overline{\{b_1=0\}}\times \A^1\minus \{0\}\to \{c_{21}=0\}$
  is a weak equivalence. The latter follows from placing it at
  the top in the commutative diagram
  \[
  \begin{tikzcd}
    \overline{\{b_1=0\}}\times \A^1\minus \{0\}\ar[r] \ar[d,"\pr_2"] &
    \{c_{21}=0\}\ar[d,"c_{22}"] \\
    \A^1\minus\{0\}\ar[r,"\id"] & \A^1\minus \{0\}
  \end{tikzcd}
  \]
  where the vertical projections are weak equivalences.
  The open complement
  $\{b_1\neq 0\}\hookrightarrow Q_3$ 
  contains the closed subscheme
  $\{b_2=0\}\hookrightarrow Q_3$ 
  as a strong $\A^1$-deformation retract, as the map
  \[ \{b_1\neq 0\} \times \A^1\to \{b_1\neq 0\},\,
  \bigl((a,b),t\bigr) \mapsto (ta_1+(1-t)b_1^{-1},a_2,b_1,tb_2) \]
  shows. This strong $\A^1$-deformation retraction is $\G$-equivariant,
  whence also the inclusion
  $\overline{\{b_2=0\}}\hookrightarrow \overline{\{b_1\neq 0\}}$
  is a strong $\A^1$-deformation retract. Note that
  $\overline{\{b_2=0\}}$ is isomorphic to $\A^1$. 
  Homotopy purity \cite[Theorem 3.2.23]{mv} supplies a homotopy cofiber sequence
  \[ \A^1 \iso \overline{\{b_2=0\}}\simeq
  \overline{\{b_1\neq 0\}} \hookrightarrow Q_3/\G
  \to \Sigma^{1+(1)}\{b_1=0\}_+ \]
  inducing a homotopy cofiber sequence after applying $\Sigma^{(1)}$.
  In particular, since $\overline{\{b_2=0\}}$ is $\A^1$-contractible,
  the map $\Sigma^{(1)}Q_3/\G\to \Sigma^{(1)}\Sigma^{1+(1)}\{b_{1}=0\}_+$
  is a weak equivalence.
  
  Similarly, the smooth closed subscheme
  $\{c_{21}=0\}\rclosed \SL_2$
  gives rise, via homotopy purity, to a homotopy cofiber sequence
  \[ \{c_{21}\neq 0\} \ropen
  \SL_2 \to \Sigma^{1+(1)}\{c_{21}=0\}_+ \]
  which can be related to the homotopy cofiber sequence above as follows.
  While the map
  $\psi_2\colon \overline{\{b_1=0\}}\times \A^1\minus \{0\}\to \{c_{21}=0\}$
  is a weak equivalence, $\psi_2(\{b_1\neq 0\}\times \A^1\minus \{0\})$ is
  not contained in $\{c_{21}\neq 0\}$ but instead coincides with the
  union $U:=\{c_{21}\neq 0\}\cup \{c_{11}=c_{22}=1 \ \mathrm{and}\ c_{21}=0\}$.
  The strong $\A^1$-deformation retraction
  \[ \{c_{21}\neq 0\} \times \A^1 \to \{c_{21}\neq 0 \},\,(C,t)\mapsto
  \begin{pmatrix}
    c_{11}+t(1-c_{11}) &
    c_{12} +t(1-c_{11})\tfrac{c_{22}}{c_{21}} \\ c_{21} & c_{22}
  \end{pmatrix} \]
  extends via the constant $\A^1$-homotopy to a
  strong $\A^1$-deformation retration of $U$ to
  $V:=\{C\in U\colon c_{11}=1\}$, which in turn
  deforms via 
  \[ V\times \A^1\to V,\, (C,t)\mapsto
  \begin{pmatrix}
    c_{11}=1 & c_{12} \\
    tc_{21} & tc_{22}+1-t 
  \end{pmatrix} \]
  to the affine line $\{c_{11}=c_{22}=1 \ \mathrm{and}\ c_{21}=0\}\rclosed \SL_2$.
  The induced diagram of homotopy cofiber sequences
  \begin{equation*}
    \begin{tikzcd}
      \{c_{21}\neq 0\} \ar[r] \ar[d] &
      \SL_2 \ar[r] \ar[d] &
      \Sigma^{1+(1)}\{c_{21}=0\}_+  \ar[d]\\
      \A^1\simeq U \ar[r] & \SL_2 \ar[r,"\simeq"] & \SL_2/U
    \end{tikzcd}
  \end{equation*}
  identifies $\Sigma^{1+(1)}\{c_{21}=0\}_+/\Sigma^{1+(1)}\{1\}_+\simeq \SL_2/U$
  and induces a commutative diagram
  \begin{equation*}
    \begin{tikzcd}
      \Sigma^{(1)}Q_3/\G\ar[r,"\simeq"] \ar[d,"\psi_2"] &
      \Sigma^{1+(2)}\overline{\{b_1=0\}}_+ \ar[d,"\simeq"] \\
      \SL_2 \ar[r,"\simeq"] & \Sigma^{1+(1)}\{c_{21}=0\}_+/\Sigma^{1+(1)}\{1\}_+
    \end{tikzcd}
  \end{equation*}
  in which the vertical map on the right hand side is a weak equivalence,
  because it is induced by the isomorphism
  $\psi_2\colon \overline{\{b_1=0\}}\times \A^1\minus \{0\} \xrightarrow{\iso}
  \{c_{21}=0\} $. Hence $\psi_2$ is a weak equivalence as claimed.
\end{proof}

The proof of the following statement is essentially a modification
of Jean Fasel's unpublished proof for the corresponding statement on symplectic
groups; I thank him sincerely for the inspiration.

\begin{proposition}\label{prop:hocofib-sln}
  The inclusion $\SL_n\rclosed \SL_{n+1}$ fits into a
  homotopy cofiber sequence
  \[ \SL_n\rclosed \SL_{n+1}\to \Sigma^{n+(n+1)}(\SL_n)_+ \]
  over $\Spec(\ZZ)$.
\end{proposition}

\begin{proof}
  A matrix $C\in \SL_{n}$ has entries denoted
  $c_{1,1},\dotsc,c_{1,n},c_{2,1},\dotsc,c_{n,n}$.
  The homotopy purity theorem  \cite[Theorem 3.2.23]{mv},
  applied to the smooth closed subscheme
  $W:=\{c_{n+1,1}=\dotsc=c_{n+1,n}=0\}\rclosed \SL_{n+1}$
  (which is a global complete intersection 
  and thus has a trivial normal bundle), 
  supplies a homotopy cofiber
  sequence:
  \[ \SL_{n+1}\minus W \ropen \SL_{n+1} \to \Sigma^{n+(n)} W_+\]
  The inclusion $(\SL_{n+1}\minus W)\cap \{c_{n+1,n+1}=1\}
  \hookrightarrow \SL_{n+1}\minus W$ is a weak equivalence.
  Let $q\colon \SL_{n+1}\minus W\to \A^n\minus \{0\}$ send $C$ to
  $(c_{n+1,1},\dotsc,c_{n+1,n})$. Let $X$ and $Y$ be defined by
  taking pullbacks
  \begin{equation}\label{eq:pb-deform-2}
    \begin{tikzcd}
      X\ar[r] \ar[d,"\simeq"] & Y\ar[r]\ar[d,"\simeq"] & Q_{2n-1} \ar[d,"\pr_1"] \\
      (\SL_{n+1}\minus W)\cap\{c_{n+1,n+1}=1\} \ar[r] & \SL_{n+1}\minus W\ar[r,"q"] & \A^n\minus \{0\}
    \end{tikzcd}
  \end{equation}
  where the vertical maps are Zariski locally trivial fibrations
  with $\A^{n-1}$ as fiber, and hence weak equivalences.
  An element in $Y$
  is a pair $(E,d)$ with $E\in \SL_{n+1}\minus W$ and
  $d\in \A^n\minus \{0\}$ such that $\sum_{j=1}^n e_{n+1,j}\cdot d_j=1$.
  The map
  \[ Y\times \A^1\to Y,\, ((E,d),t) \mapsto
    E\cdot \Bigl(\begin{pmatrix}
      1_n & t(1-e_{n+1,n+1})d \\
      0  & 1 
    \end{pmatrix},d\Bigr) \]
  is a strong $\A^1$-deformation retraction onto $X$. Moreover,
  \[  (E,t)\mapsto
    \begin{pmatrix} 1_n & 
      -t E_{j,n+1} \\ 0 & 1
    \cdot E \end{pmatrix}\]
  is an $\A^1$-deformation retraction of $(\SL_{n+1}\minus W)\cap\{c_{n+1,n+1}=1\}$
  onto the closed subscheme
  \[\SL_n\times \bigl(\A^n\minus \{0\}\bigr) \hookrightarrow
    (\SL_{n+1}\minus W)\cap\{c_{n+1,n+1}=1\},\,(E,d)\mapsto
    \begin{pmatrix} E & 0 \\ d & 1 
    \end{pmatrix}.\]
  As in the proof of Lemma~\ref{lem:gmap2-weq},
  the homotopy purity theorem provides a homotopy cofiber sequence
  \begin{equation}\label{eq:hocof-2}
    \SL_n\times \bigl(\A^n\minus \{0\}\bigr) \hookrightarrow \SL_{n+1}
    \to \Sigma^{n+(n)} W_+
  \end{equation}
  in which $W = \{c_{n+1,1}=\dotsc=c_{n+1,n}=0\}
  \iso \SL_n\times (\A^1\minus \{0\})\times \A^n\simeq \SL_n\times (\A^1\minus \{0\})$.
  Enlarging the subscheme $\SL_n\times \bigl(\A^n\minus \{0\}\bigr)$
  in~(\ref{eq:hocof-2}) to $\SL_n\times \A^n$ then provides the desired
  homotopy cofiber sequence:
  \[ \SL_n\simeq \SL_n\times \A^n \hookrightarrow \SL_{n+1}
    \to \Sigma^{n+(n+1)} (\SL_n)_+\]
\end{proof}

\begin{corollary}\label{cor:psi-hocof-hofib}
  In the commutative diagram
  \begin{center}
    \begin{tikzcd}
      \Sigma^{(1)}Q_{2n-1}/\G \ar[r] \ar[d,"\psi_n"] & \Sigma^{(1)}Q_{2n+1}/\G \ar[r]\ar[d,"\psi_{n+1}"] &
      S^{n+(n+1)} \ar[d] \\
      \SL_{n}\ar[r] &\SL_{n+1}\ar[r] & \SL_{n+1}/\SL_{n}
    \end{tikzcd}
  \end{center}
  in which the top row is a homotopy cofiber sequence and the
  bottom row is a homotopy fiber sequence,
  the canonically induced map $S^{n+(n+1)}\to \SL_{n+1}/\SL_n$ is
  a weak equivalence over $\Spec(\ZZ)$.
\end{corollary}

\begin{proof}
  Proposition~\ref{prop:hocofib-sln} and
  the standard homotopy cofiber sequence for
  $Q_{2n-1}/\G\rclosed Q_{2n+1}/\G$ give a diagram of homotopy cofiber sequences
  \begin{center}
    \begin{tikzcd}
      \Sigma^{(1)}Q_{2n-1}/\G \ar[r] \ar[d,"\psi_n"] & \Sigma^{(1)}Q_{2n+1}/\G \ar[r]\ar[d,"\psi_{n+1}"] &
      \Sigma^{(1)}\PP^n/\PP^{n-1} = \Sigma^{n+(n+1)}(\Spec(\ZZ))_+ \ar[d] \\
      \SL_{n}\ar[r] &\SL_{n+1}\ar[r] & \Sigma^{n+(n+1)}(\SL_n)_+
    \end{tikzcd}
  \end{center}
  where the vertical map on the right hand side is induced by the inclusion
  of the identity matrix. The canonical
  map $\Sigma^{n+(n+1)}(\SL_n)_+\to \SL_{n+1}/\SL_n$
  is induced by the structure map $\SL_n\to \Spec(\ZZ)$,
  because the pullback square
  \begin{center}
    \begin{tikzcd}
      W\ar[r] \ar[d,"\ell"] & \SL_{n+1} \ar[d,"\mathrm{last\ row}"] \\
      \{0\}\times (\A^1\minus \{0\}) \ar[r] & \A^{n+1}\minus \{0\}
    \end{tikzcd}
  \end{center}
  of smooth schemes induces a commutative diagram
  \begin{center}
    \begin{tikzcd}
      (\A^n/\A^n\minus \{0\})\smash W_+ \ar[r,"\simeq"] \ar[d,"(\A^n/\A^n\minus \{0\})\smash \ell_+"] &
      \SL_{n+1}/(\SL_{n+1}\minus W) \ar[d] \\
      (\A^n/\A^n\minus \{0\})\smash (\A^1\minus \{0\})_+ \ar[r,"\simeq"] &
      \A^{n+1}\minus \{0\} /(\A^n\minus \{0\})\times \A^1
    \end{tikzcd}
  \end{center}
  of homotopy purity transformations, where the map $\ell$ sends
  a matrix in $W\simeq \SL_n\times (\A^1\minus \{0\})$ to its last diagonal element
  and hence is induced by the structure map $\SL_n\to \Spec(\ZZ)$.
  Proceeding through the zigzag relating $\SL_{n+1}/(\SL_{n+1}\minus W)$
  with $\SL_{n+1}/\SL_n$ produced in the proof of Proposition~\ref{prop:hocofib-sln}
  provides the statement.
\end{proof}

With the help of $\psi_3$, the cell structure of $\SL_3$
looks as follows.

\begin{lemma}\label{lem:cofiber-psi3}
  The homotopy cofiber of $\psi_3\colon \Sigma^{(1)}\PP^2\to \SL_3$
  over $\Spec(\ZZ)$ is given by $S^{3+(5)}$.
\end{lemma}

\begin{proof}
  Lemma~\ref{lem:gmap2-weq}, Proposition~\ref{prop:hocofib-sln}
  and elementary properties of homotopy pushout diagrams imply
  that the diagram
  \begin{equation*}
    \begin{tikzcd}
      \Sigma^{(1)}Q_{5}/\G \ar[r,"\psi_3"]\ar[d] & \SL_3 \ar[d] \\
      S^{2+(3)}\ar[r] & \Sigma^{2+(3)}(\SL_2)_+
    \end{tikzcd}
  \end{equation*}
  in which the vertical maps are the canonical quotient maps and the
  bottom horizontal map is the canonical inclusion is a homotopy
  pushout diagram. The result follows.
\end{proof}

More generally, the total homotopy cofiber of
diagram~(\ref{eq:diagram-psi}) can be determined as follows.

\begin{lemma}\label{lem:connectivity-psi-diagram}
  The homotopy cofiber of the map
  $\Sigma^{(1)}Q_{2n+1}/\G\cup_{\Sigma^{(1)}Q_{2n-1}/\G} \SL_n \to \SL_{n+1}$
  induced by the commutative diagram
  \begin{equation*}
    \begin{tikzcd}
      \Sigma^{(1)}Q_{2n-1}/\G \ar[r] \ar[d,"\psi_n"] & \Sigma^{(1)}Q_{2n+1}/\G
      \ar[d,"\psi_{n+1}"] \\
      \SL_{n}\ar[r] &\SL_{n+1}
    \end{tikzcd}
  \end{equation*}
  is equivalent to $\Sigma^{n+(n+1)}\SL_n$ over $\Spec(\ZZ)$. 
\end{lemma}

\begin{proof}
  This follows from Proposition~\ref{prop:hocofib-sln},
  Corollary~\ref{cor:psi-hocof-hofib}
  and a straightforward manipulation of homotopy pushout squares.
\end{proof}

In the following, $Q_{2n-1}/\G$ will be identified via
projection to the first $n$ coordinates with $\PP^{n-1}$,
which gives rise to maps such as
$\psi_n\colon \Sigma^{(1)}\PP^{n-1}\simeq \Sigma^{(1)}Q_{2n-1}/\G\to \SL_n$
in the homotopy category. Recall from \cite[Convention 2.3.5]{aww}
that  a map $f$ of pointed motivic spaces is $\A^1$-$n$-connected if
its homotopy fiber\footnote{One has to take homotopy fibers at all
  $\A^1$-path components in the case (which will not occur here)
  that the target motivic space is not $\A^1$-$0$-connected.}
$\hofib(f)$ is $\A^1$-$(n-1)$-connected, which is
equivalent to the homomorphism $\pia_jf$ being an isomorphism
for $j<n$ and an epimorphism for $j=n$.

\begin{lemma}\label{lem:pi1-computation}
  Let $F$ be a field. The inclusion $\Sigma^{(1)}\PP^1\to \Sigma^{(1)}\PP^2$
  induces the canonical projection
  \[ \KMW_2\iso \pia_1 \Sigma^{(1)}\PP^1\to \pia_1 \Sigma^{(1)}\PP^2\iso \KMil_2\]
  on $\A^1$-fundamental groups. The inclusion
  $\Sigma^{(1)}\PP^{n-1}\to \Sigma^{(1)}\PP^n$ is $\A^1$-$(n-1)$-connected
  for all $n>0$.
\end{lemma}

\begin{proof}
  The assumption that $F$ is perfect may be imposed by pulling back
  from a perfect subfield, over which everything in sight is defined.
  As a suspension of an $\A^1$-$0$-connected variety, $\Sigma^{(1)}\PP^n$ is
  $\A^1$-$0$-connected for all $n$ \cite[Lemma 3.3.1]{aww}.
  The determination of $\pia_1\PP^n$ from
  \cite[Theorem 7.13 and Theorem 7.29]{morel.at}
  implies that the map $\pia_1\PP^1\to \pia_1\PP^n$ is surjective for all $n>0$.
  In other words, the inclusion $\PP^1\to \PP^n$ is $\A^1$-$1$-connected for all $n>1$.
  While smashing with $\G$ preserves the simplicial connectivity of a map,
  it is a priori not clear whether smashing with $\G$ preserves the
  $\A^1$-connectivity of a map. Let $\Sing$ denote the endofunctor
  on (pointed) simplicial presheaves on $\Sm_F$
  introduced in \cite[p.~87]{mv}. It commutes with limits and colimits. Since
  the natural transformation $\G\to \Sing(\G)$ is an isomorphism, there
  results a natural isomorphism $\Sigma^{(1)}\Sing(B)\xrightarrow{\iso}
  \Sing(\Sigma^{(1)}B)$ for every pointed simplicial presheaf $B$.
  Choosing appropriate $\A^1$-naive models for projective spaces -- 
  which is possible by \cite[Example 4.2.13]{ahw.2} -- provides a model
  for the canonical inclusion $\PP^1\to \PP^n$ such that
  $\Sing(\PP^1)\to \Sing(\PP^n)$ is simplicially $1$-connected,
  and hence so is
  \begin{equation}\label{eq:gm-susp-inclusion}
    \Sigma^{(1)}\Sing(\PP^1)\iso \Sing(\Sigma^{(1)}\PP^1)\to \Sing(\Sigma^{(1)}\PP^n)\iso \Sigma^{(1)}\Sing(\PP^n).
  \end{equation}
  The source is stalkwise equivalent to $\Sing(\SL_2)$ by \cite[Lemma 4.2.4]{ahw.2}.
  Since $\SL_2$ is $\A^1$-naive by \cite[Theorem 4.2.1]{ahw.2} (applied to
  $Q_3\iso \SL_2$), there results an isomorphism
  \[ \pia_0\Omega\Sigma^{(1)}\Sing(\PP^1)\iso \pia_1\Sigma^{(1)}\PP^1\iso \pia_1\SL_2
    \iso \KMW_2\]
  where the last isomorphism follows from \cite[Theorem 1.27]{morel.at} (see also \cite[Theorem 1]{morel-sawant}).
  Thus $\pia_0\Omega\Sigma^{(1)}\Sing(\PP^1)$ is a strictly $\A^1$-invariant
  sheaf. To prove the same for $\pia_0\Omega\Sigma^{(1)}\Sing(\PP^n)$,
  let $G$ denote the (simplicial) homotopy fiber of the
  map~(\ref{eq:gm-susp-inclusion}).  
  The aforementioned connectivity of the map~(\ref{eq:gm-susp-inclusion})
  implies that $G$ is simplicially $0$-connected, and hence
  $\A^1$-$0$-connected by \cite[Cor.~2.3.22]{mv}.
  Since $\Omega\Sigma^{(1)}\Sing(\PP^n)$ is the simplicial homotopy fiber
  of the canonical map $G\to \Sigma^{(1)}\Sing(\PP^1)$, 
  \cite[Theorem 6.56]{morel.at} (see also \cite[Corollary 2.3.6]{aww})
  provides an exact sequence
  \[ \pia_1 G\to \pia_1\Sigma^{(1)}\Sing(\PP^1) \to \pia_0\Omega\Sigma^{(1)}\Sing(\PP^n)\to \pia_0G= 0 \]
  using that $\pia_0\Omega\Sigma^{(1)}\Sing(\PP^1)$ is a strictly
  $\A^1$-invariant sheaf.
  In this exact sequence $\pia_1G$ is strongly $\A^1$-invariant by
  \cite[Theorem 6.1]{morel.at}, its image in
  $\pia_1\Sigma^{(1)}\Sing(\PP^1)$ is strongly $\A^1$-invariant
  by~\cite[Theorem 1.6]{choudhury-hogadi}, and moreover abelian,
  and the cokernel 
  $\pia_0\Omega\Sigma^{(1)}\Sing(\PP^n)$ is then a strictly $\A^1$-invariant
  sheaf by \cite[Corollary 6.24]{morel.at}. As a consequence
  \cite[Theorem 6.56]{morel.at} 
  applies to show that
  $\Sigma^{(1)}\PP^1\to \Sigma^{(1)}\PP^n$ is $\A^1$-$1$-connected.

  To determine $\pia_1\Sigma^{(1)}\PP^n$ for $n>1$, 
  Morel's $\A^1$-Hurewicz
  theorem \cite[Theorem 6.35]{morel.at} implies that
  the Hurewicz transformation 
  $\pia_1\Sigma^{(1)}\PP^n \to \underline{H}_1^{\A^1}\Sigma^{(1)}\PP^n$ is an
  isomorphism for all $n$. For $n=2$,
  the latter can be determined via the cofiber
  sequence
  \[ S^{1+(3)} \xrightarrow{\eta} \Sigma^{(1)}\PP^{1}\to \Sigma^{(1)}\PP^2 \]
  as $\underline{H}_1^{\A^1}\Sigma^{(1)}\PP^2\iso \KMil_2$.
  Similarly, for $n>2$ the inclusion
  $\PP^{n-1} \hookrightarrow \PP^n$
  is $\A^1$-$(n-1)$-connected, as one may deduce from the
  homotopy fiber of the ``covering'' map $\A^n\minus \{0\}\hookrightarrow
  \A^{n+1}\minus \{0\}$ of universal $\A^1$-coverings.
  Again smashing with $\G$ preserves the simplicial connectivity of
  $\PP^{n-1} \hookrightarrow \PP^n$. Arguing with $\A^1$-naive models
  and the functor $\Sing$
  as before provides that $\Sigma^{(1)}\PP^{n-1}\to \Sigma^{(1)}\PP^n$
  is $\A^1$-$(n-1)$-connected, again invoking \cite[Theorem 6.56]{morel.at}
  or \cite[Corollary 2.3.6]{aww} and the strict $\A^1$-invariance of
  $\pia_0\Omega\Sigma^{(1)}\Sing(\PP^n)$ already established.
\end{proof}

The pushout in diagram~(\ref{eq:diagram-psi}) gives rise to a map
\begin{equation}\label{eq:theta}
  \theta_{n+1}\colon \Sigma^{(1)}\PP^n\cup_{\Sigma^{(1)}\PP^{n-1}} \SL_n \to \SL_{n+1}
\end{equation}
which factors $\psi_{n+1}\colon \Sigma^{(1)}\PP^n\to \SL_{n+1}$ for every $n>0$.

\begin{proposition}\label{prop:image-diagram-psi}
  Let $F$ be a field.
  The image of the homomorphism
  $\pia_n\SL_{n+1}\to \pia_nS^{n+(n+1)}$ induced
  by taking the last column of a matrix is isomorphic to the
  image of the homomorphism
  $\pia_n\Sigma^{(1)}\PP^n\cup_{\Sigma^{(1)}\PP^{n-1}}\SL_n\to \pia_nS^{n+(n+1)}$
  induced by the canonical quotient map collapsing $\SL_n$ to the basepoint.
\end{proposition}

\begin{proof}
  In case $n=1$, both the last column map and the quotient map are
  equivalences, whence the induced homomorphisms are
  isomorphisms. Let $n>1$.
  Diagram~(\ref{eq:diagram-psi}) and the map $\theta_{n+1}$
  defined in~(\ref{eq:theta}) induce the following
  commutative diagram
  \begin{equation}\label{eq:hocof-hofib-suslin}
    \begin{tikzcd}
      \pia_n \Sigma^{(1)}\PP^{n-1} \ar[r] \ar[d,"\pia_n\psi_n"] & \pia_n\Sigma^{(1)}\PP^{n} \ar[r]\ar[d] &
      \pia_n S^{n+(n+1)}  \ar[d,"\iso"] \\
      \pia_n \SL_n \ar[r] \ar[d,"\id"] \ar[r] & \pia_n\Sigma^{(1)}\PP^{n}\cup_{\Sigma^{(1)}\PP^{n-1}}\SL_n \ar[r]\ar[d,"\pia_n\theta_{n+1}"] &
      \pia_n S^{n+(n+1)}  \ar[d,"\iso"] \\
      \pia_n\SL_{n}\ar[r] &\pia_n\SL_{n+1} \ar[r] & \pia_nS^{n+(n+1)} %\ar[r] &
    \end{tikzcd}
  \end{equation}
  of sheaves of $\A^1$-homotopy groups, in which the vertical homomorphisms
  on the right hand side are isomorphisms by
  Corollary~\ref{cor:psi-hocof-hofib} and by construction.
  In particular, the image of the homomorphism
  $\pia_n\Sigma^{(1)}\PP^n\cup_{\Sigma^{(1)}\PP^{n-1}}\SL_n\to \pia_nS^{n+(n+1)}$
  embeds in the image of the homomorphism
  $\pia_n\SL_{n+1}\to \pia_nS^{n+(n+1)}$.
  To prove that the image of
  $\pia_n\SL_{n+1}\to \pia_nS^{n+(n+1)}$
  coincides with the image of
  $\pia_n\Sigma^{(1)}\PP^n\cup_{\Sigma^{(1)}\PP^{n-1}} \SL_n \to \pia_n S^{n+(n+1)}$
  induced by collapsing $\SL_n$, it suffices to prove that the
  homomorphism
  $\pia_n\theta_{n+1}\colon \pia_n\Sigma^{(1)}\PP^n\cup_{\Sigma^{(1)}\PP^{n-1}} \SL_n \to
  \pia_n\SL_{n+1}$
  is surjective. This in turn follows if the  map
  $\theta_{n+1}$
  is $\A^1$-$n$-connected. This is the connectivity of its
  homotopy cofiber $\cone(\theta_{n+1})$
  by Lemma~\ref{lem:connectivity-psi-diagram}.
  Unfortunately this does not necessarily imply that its homotopy fiber
  $\hofib(\theta_{n+1})$
  is $\A^1$-$(n-1)$-connected. To conclude this nevertheless, start with the case
  $n=2$. Then the map $\theta_3$ in question
  coincides with $\psi_3\colon \Sigma^{(1)}\PP^2\to \SL_3$ up to equivalence
  by Lemma~\ref{lem:gmap2-weq}. Lemma~\ref{lem:pi1-computation}
  implies together with the determination of $\pia_1\SL_n$ from
  \cite[Theorem 1]{morel-sawant} that $\pia_1\psi_3$ is an isomorphism. There
  results an exact sequence
  \[ \pia_2\Sigma^{(1)}\PP^2\xrightarrow{\pia_2\psi_3} \pia_2\SL_3\to \pia_1\hofib(\psi_3) \to 0 \]
  whence it remains to prove that $\pia_2\psi_3$ is an epimorphism.
  Consider the commutative diagram
  \begin{equation}\label{eq:bottom-exact}
    \begin{tikzcd}
      \pia_2\Sigma^{(1)}\PP^1\ar[r] \ar[d,"\psi_2"] &
      \pia_2\Sigma^{(1)}\PP^2\ar[r] \ar[d,"\psi_3"] &
      \pia_2S^{2+(3)} \ar[r] \ar[d,"\iso"] &
      \pia_1\Sigma^{(1)}\PP^1\ar[r] \ar[d,"\psi_2"]  &
      \pia_1\Sigma^{(1)}\PP^2\ar[r] \ar[d,"\psi_3"] & 0 \\
      \pia_2\SL_2\ar[r]  &
      \pia_2\SL_3\ar[r] &
      \pia_2\A^{3}\minus \{0\} \ar[r]  &
      \pia_1\SL_2\ar[r] &
      \pia_1\SL_3\ar[r]  & 0
    \end{tikzcd}
  \end{equation}
  induced by diagram~(\ref{eq:diagram-psi}) in which
  the bottom row is induced by the homotopy fiber sequence
  \[ \SL_2\to \SL_3 \xrightarrow{\mathrm{last}\ \mathrm{column}} \A^{3}\minus \{0\} \]
  and in particular exact. The top row in diagram~(\ref{eq:bottom-exact})
  is induced by the homotopy cofiber sequence
  \[ \Sigma^{(1)}\PP^1\to \Sigma^{(2)}\PP^2\to S^{2+(3)}\]
  and the weak equivalence $\psi_2\colon \Sigma^{(1)}\PP^1\to \SL_2$
  from Lemma~\ref{lem:gmap2-weq}, in the sense that the homomorphism
  $\pia_2S^{2+(3)}\to \pia_1\Sigma^{(1)}\PP^1$ is the composition of
  the inverse of $\pia_1\psi_2$ and
  $\pia_2S^{2+(3)}\iso \pia_2\A^3\minus \{0\}\to \pia_1\SL_2$.
  Here the first isomorphism follows from Corollary~\ref{cor:psi-hocof-hofib}.
  It follows that the upper row in diagram~(\ref{eq:bottom-exact})
  is exact at the spots involving $\pia_1$. To conclude
  exactness at $\pia_2S^{2+(3)}$, it remains to prove that the image
  of $\pia_2\Sigma^{(1)}\PP^2\to \pia_2S^{2+(3)}$ contains the kernel
  of $\pia_2S^{2+(3)}\to \pia_1\Sigma^{(1)}\PP^1$. By construction
  of the homotopy cofiber sequence, the latter coincides
  with the kernel of $\eta\colon \KMW_3\to \KMW_2$, which
  is $\hyper\KMW_3$ by Theorem~\ref{thm:ker-eta-pi0}.
  The map
  \[ \PP^1\times \PP^1\to \PP^2,\,\bigl((a_0:a_1),(b_0:b_1)\bigr)\mapsto (a_0b_0:a_0b_1+a_1b_0:a_1b_1)\]
  induces a commutative diagram
  \begin{equation*}
    \begin{tikzcd}
      \PP^1\times \PP^1 \ar[r] \ar[d,"\mathrm{can.}"] & \PP^2 \ar[d,"\mathrm{can.}"] \\
      \PP^1\smash \PP^1 \ar[r,"\hyper"] & \PP^2/\PP^1
    \end{tikzcd}
  \end{equation*}
  in which the identification of the lower horizontal map
  can be deduced from $\PP^1$-stabilizing first, then observing that its
  $\PP^1$-stabilization is in the kernel of multiplication with $\eta$ on the
  Grothendieck-Witt ring (hence an integer multiple of $\hyper$ by
  Theorem~\ref{thm:ker-eta-pi0}),  and finally a degree argument
  using motivic cohomology or realization,
  showing that the rank of the integer multiple of $\hyper$ is $2$.
  Morel's Theorem \cite[Theorem 1.23]{morel.at}
  then implies that the image of the homomorphism
  $\pia_2\Sigma^{(1)}\PP^1\times \PP^1\to \pia_2\Sigma^{(1)}\PP^2\to \pia_2S^{2+(3)}$
  contains $\hyper\KMW_3$. In particular, the top row in
  diagram~(\ref{eq:bottom-exact}) is also exact at $\pia_2 S^{2+(3)}$.
  Together with the isomorphism $\pia_2\psi_2$ from Lemma~\ref{lem:gmap2-weq},
  a diagram chase then provides that $\pia_2\psi_3$ is an epimorphism.  
  
  Diagram~(\ref{eq:diagram-psi}) furthermore implies that $\psi_{n+1}$ is at least
  $\A^1$-2-connected for all $n>1$.
  Invoking the $\A^1$-van Kampen theorem \cite[Theorem 3.10]{wendt.toric},
  \cite[Theorem 7.12]{morel.at} provides
  with Lemma~\ref{lem:pi1-computation} that $\pia_1\theta_{n+1}$
  is an isomorphism for all $n>1$.
  Hence 
  $\theta_{n+1}\colon \Sigma^{(1)}\PP^n\cup_{\Sigma^{(1)}\PP^{n-1}} \SL_n \to \SL_{n+1}$
  is $\A^1$-2-connected as well.
  To reach further, let $n>2$ and consider the transformation
  \begin{equation}\label{eq:fiber-seq-trans}
    \begin{tikzcd}
      \Omega S^{n+(n+1)} \ar[r] \ar[d] & \SL_n \ar[r] \ar[d] & \SL_{n+1} \ar[d,"\id"] \\
      \hofib(\theta_{n+1}) \ar[r] & \Sigma^{(1)}\PP^n\cup_{\Sigma^{(1)}\PP^{n-1}} \SL_n
      \ar[r,"\theta_{n+1}"] & \SL_{n+1}
    \end{tikzcd}
  \end{equation}
  of homotopy fiber sequences, inducing a transformation of long exact
  sequences of homotopy groups. The vertical map in the middle
  of diagram~(\ref{eq:fiber-seq-trans}) is $\A^1$-$(n-1)$-connected.
  In fact, as the cobase change of the map
  $\Sigma^{(1)}\PP^{n-1}\to \Sigma^{(1)}\PP^n$ which is $\A^1$-$(n-1)$-connected
  by Lemma~\ref{lem:pi1-computation}, it is simplicially $(n-1)$-connected.
  To apply \cite[Theorem 6.56]{morel.at} or \cite[Corollary 2.3.6]{aww} and
  conclude the desired $\A^1$-connectivity, it
  remains to prove that
  $\pia_0\Omega \Sigma^{(1)}\PP^n\cup_{\Sigma^{(1)}\PP^{n-1}}\SL_n$
  is strongly $\A^1$-invariant. The simplicial van Kampen theorem
  \cite[Corollary 3.5]{wendt.toric} provides that the Nisnevich sheaf
  associated with the presheaf of fundamental groups of
  $\Sing (\Sigma^{(1)}\PP^n)\cup_{\Sing(\Sigma^{(1)}\PP^{n-1})}\Sing(\SL_n)$
  (using $\A^1$-naiveness provided by \cite[Theorem 8.1]{morel.at} for $\GL_n$
  and \cite[Example 4.2.13]{ahw.2})
  coincides with $\KMil_2$. In particular, it is strictly $\A^1$-invariant.

  The long exact sequence diagram~(\ref{eq:fiber-seq-trans}) induces on
  $\A^1$-homotopy groups then provides epimorphisms
  $\pia_{j+1}S^{n+(n+1)} \iso \pia_j\Omega S^{n+(n+1)}\to \pia_j \hofib(\theta_{n+1})$
  for all $j<n$. In particular, Morel's $\A^1$-connectivity for $S^{n+(n+1)}$
  provides $\pia_j\hofib(\theta_{n+1})=0$ for $j<n-1$, so that $\theta_{n+1}$
  is at least $\A^1$-$(n-1)$-connected. To conclude the vanishing
  of $\pia_{n-1}\hofib(\theta_{n+1})$, observe that it is a quotient of
  $\pia_nS^{n+(n+1)}\iso \KMW_{n+1}$. The $\A^1$-homotopy
  groups induced by diagram~(\ref{eq:fiber-seq-trans}) are modules
  over $\pia_1\SL_n\iso \pia_1(\Sigma^{(1)}\PP^n\cup_{\Sigma^{(1)}\PP^{n-1}}\SL_n)
  \iso \KMil_2$, and all homomorphisms involved are $\KMil_2$-equivariant.
  However, the action of $\pia_1\SL_n$ on $\pia_j \Omega S^{n+(n+1)}$ is
  trivial because it factors through $\pia_1\ast$, as the commutative diagram
  \[
    \begin{tikzcd}
      \Omega S^{n+(n+1)} \ar[r] \ar[d,"\id"] & \SL_n \ar[r]\ar[d] & \SL_{n+1}\ar[d] \\
      \Omega S^{n+(n+1)} \ar[r] & \ast \ar[r] & S^{n+(n+1)}
    \end{tikzcd}
  \]
  of fiber sequences implies.
  Hence the action of $\pia_1(\Sigma^{(1)}\PP^n\cup_{\Sigma^{(1)}\PP^{n-1}}\SL_n)$
  on the quotient $\pia_j\hofib(\theta_{n+1})$ of $\pia_j\Omega S^{n+(n+1)}$
  is also trivial.
  It follows that the relative Hurewicz homomorphism
  $\pia_{n-1}\hofib(\theta_{n+1})\to \underline{H}_{n}^{\A^1}\cone(\theta_{n+1})$
  introduced in \cite[Section 4.2]{afh.localization}
  is an isomorphism by \cite[Theorem 4.2.1]{afh.localization}. The
  latter group is trivial by Lemma~\ref{lem:connectivity-psi-diagram}.
  Hence $\theta_{n+1}$ is $\A^1$-$n$-connected.
\end{proof}

\begin{theorem}\label{thm:suslin-conj}
  Let $F$ be an infinite field of characteristic different from $2$ and $3$,
  and $A$ an essentially smooth local $F$-algebra.
  The image of the Suslin-Hurewicz homomorphism
  $K^{\mathrm{Quillen}}_4(A) \to \KMil_4(A)$ coincides with $6\KMil_4(A)$.
\end{theorem}

\begin{proof}
  By \cite[Theorem 2.18]{afw} it suffices to prove that a certain
  canonical surjection $\KMil_4/6 \to \sheafS$, obtained from Suslin's
  Hurewicz homomorphism, is an isomorphism.
  One description of $\sheafS$ goes as follows. The $\A^1$-fiber sequence
  \[ S^{3+(4)}\to \BSL_3\to \BSL_4 \]
  of motivic spaces
  induces a long exact sequence of unstable homotopy sheaves terminating with
  \begin{equation}\label{eq:les-sheafs} \dotsm \to \pia_3\SL_4\to
    \pia_3 S^{3+(4)} \to \pia_3\BSL_3\to \pia_3\BSL_4 \to 0.
  \end{equation}
  The sheaf $\sheafS$ is isomorphic to the kernel of $\pia_3\BSL_3\to \pia_3\BSL_4$
  by \cite[Lemma 3.5]{af.bundles-spheres}.
  Homotopical stability provides that
  $\pia_3\BSL_4\iso \pia_3\BSL_\infty\iso \Kalg_3$
  is the third algebraic $K$-theory
  sheaf, as explained in \cite{afw}. Morel's unstable computations
  provide $\pia_3S^{3+(4)}\iso \KMW_4$ \cite[Theorem 1.23]{morel.at}.
  Hence there exists
  a surjection $\KMW_4\iso \pia_3S^{3+(4)}\to \sheafS$, which, as is
  also explained in \cite{afw}, factors over $\KMil_4/6$.
  In the case where $F$ is a field of characteristic $p>3$, the
  map $\KMil_4/6 \to \KMil_4[p^{-1}]/6$ is injective by \cite{izhboldin.torsion}.
  Therefore it
  suffices to prove that the induced map $\KMil_4[p^{-1}]/6\to \sheafS[p^{-1}]$
  is injective. In the following, the characteristic $p$ of the field $F$ may be
  implicitly inverted if $p>3$.
  
  Using the exact sequence~(\ref{eq:les-sheafs}), the image of
  $\pia_3\SL_4\to \pia_3S^{3+(4)}$ (induced by taking the last column
  of a matrix in $\SL_4$) contains at least the subsheaves
  $6\KMW_4$ and $\eta\KMW_5$, and hence their sum.
  Proposition~\ref{prop:image-diagram-psi} shows that this image
  coincides with the image
  of $\pia_3\Sigma^{(1)}\PP^3\cup_{\Sigma^{(1)}\PP^2}\SL_3\xrightarrow{\pia_3\theta_4} \pia_3\SL_4\to \pia_3S^{3+(4)}$.
  Passage to motivic suspension spectra simplifies the
  situation via the following commutative diagram
  \begin{equation*}
    \begin{tikzcd}
      \pia_3\Sigma^{(1)}\PP^3\cup_{\Sigma^{(1)}\PP^2}\SL_3\ar[r] \ar[d] &  \pia_3S^{3+(4)} \ar[d] \\
      \pistablesheaf_{3+(0)}\Sigma^{(1)}\PP^3\cup_{\Sigma^{(1)}\PP^2}\SL_3\ar[r] &
      \pistablesheaf_{3+(0)}S^{3+(4)}
    \end{tikzcd}
  \end{equation*}
  in which the vertical homomorphisms are induced by taking $\PP^1$-suspension
  spectra. In particular, the vertical homomorphism
  on the right hand side is an isomorphism
  by Morel's theorem.
  Hence the image in question is contained in the image of
  the homomorphism
  $\pistablesheaf_{3+(0)}\Sigma^{(1)}\PP^3\cup_{\Sigma^{(1)}\PP^2}\SL_3\to
  \pistablesheaf_{3+(0)}S^{3+(4)}$
  which coincides with the kernel of
  the homomorphism
  $\pistablesheaf_{3+(0)}S^{3+(4)} \to \pistablesheaf_{2+(0)} \SL_3$
  because of the cofiber sequence
  \[ \SL_3\to \Sigma^{(1)}\PP^3\cup_{\Sigma^{(1)}\PP^2} \SL_3 \to S^{3+(4)}. \]
  Using the map $\psi_3\colon \Sigma^{(1)}\PP^2\to \SL_3$,
  this kernel contains the kernel of
  $\pistablesheaf_{3+(0)}S^{3+(4)} \to \pistablesheaf_{2+(0)} \Sigma^{(1)}\PP^2$.
  These kernels are equal, because the map
  $\pistablesheaf_{2+(0)}\psi_3 \colon \pistablesheaf_{2+(0)}\Sigma^{(1)}\PP^2\to
  \pistablesheaf_{2+(0)}\SL_3$
  is not only surjective (a straightforward consequence of
  Lemma~\ref{lem:cofiber-psi3}), but injective as well (at least after inverting
  the exponential characteristic of $F$ if it is odd).
  The latter can be seen from Lemma~\ref{lem:cofiber-psi3}
  which supplies a homotopy cofiber sequence 
  \[ \Sigma^{(1)}\PP^2\xrightarrow{\psi_3}\SL_3\to S^{3+(5)}. \]
  It induces a map $S^{3+(5)}\to \Sigma^{1+(1)}\PP^2$ whose composition
  with the canonical quotient map $\Sigma^{1+(1)}\PP^2\to S^{3+(3)}$
  gives rise to an element in the Witt ring $\pistablesheaf_{3+(5)}S^{3+(3)}\iso
  \KMW_{-2}$ by Morel's theorem.
  This element
  is the zero element, because the image of
  $\pistablesheaf_{3+(5)}\Sigma^{1+(1)}\PP^2\to \pistablesheaf_{3+(5)}S^{3+(3)}$
  is zero by the short exact sequence~(\ref{eq:pi2p2}). Hence
  the map $S^{3+(5)}\to \Sigma^{1+(1)}\PP^2$ factors over the inclusion
  $\Sigma^{1+(1)}\PP^1\to \Sigma^{1+(1)}\PP^2$ and gives rise to
  an element in the group
  $\pistablesheaf_{3+(5)}S^{2+(2)} =\pistablesheaf_{1+(3)}\unit_F$
  which is trivial by \cite[Theorem 5.5]{rso.oneline}
  after inverting the characteristic of $F$ if it is odd.

  To summarize, the image of the homomorphism
  $\pia_3\SL_4\to \pia_3S^{3+(4)}$ of unstable $\A^1$-homotopy sheaves
  is contained in the kernel of the homomorphism of stable $\A^1$-homotopy
  sheaves
  $\pistablesheaf_{3+(0)}S^{3+(4)} \to \pistablesheaf_{2+(0)} \Sigma^{(1)}\PP^2$.
  Using $\PP^1$-stability, the latter homomorphism can be identified with
  the homomorphism
  $\pistablesheaf_{2-(1)}S^{2+(3)} \to \pistablesheaf_{2-(1)} \PP^2$
  induced by the attaching map $\gamma_2\colon \A^{3}\minus \{0\}\to \PP^2$
  of smooth schemes because of the cofibration sequence
  \[ \A^{3}\minus \{0\}\xrightarrow{\gamma_2}\PP^2 \hookrightarrow \PP^3.\]
  The composition
  \[ S^{2+(3)}\xrightarrow{\gamma_2}\PP^2\xrightarrow{q}S^{2+(2)} \] is
  nullhomotopic essentially because $2$ is even. 
  More precisely, by Morel's Theorem, it suffices to prove this after applying
  $\Sigma^{(1)}$.
  The composition of the attaching map $\Sigma^{(1)} \gamma_2$ with the
  canonical inclusion to $\Sigma^{(1)}\PP^3$ is nullhomotopic by construction,
  hence also the composition with $\psi_4\colon \Sigma^{(1)}\PP^3\to \SL_4$.
  By the commutative diagram
  appearing in Corollary~\ref{cor:psi-hocof-hofib} and exactness,
  the composition $\psi_3\circ \Sigma^{(1)}\gamma_2$ is in the image of
  the connecting map $\delta_3\colon \pia_{3+(4)}S^{3+(4)}\to \pia_{2+(4)}\SL_3$.
  The result follows from \cite[Lemma 3.5]{af.bundles-spheres} and
  Corollary~\ref{cor:psi-hocof-hofib}.
  Hence the attaching map $\gamma_2$ corresponds $\PP^1$-stably to an element
  in the image of $\pi_{2+(3)}\PP^1\to \pi_{2+(3)}\PP^2$. 
  This element is $\gamma_2=\pm 2(i\nu)\in \pi_{2+(3)}\PP^2\iso \ZZ/12\{i\nu\}$
  by Remark~\ref{rem:unstable}. It follows that the kernel
  of the induced homomorphism
  \[ \KMW_\star\iso \pistablesheaf_{2-(\star-3)}S^{2+(3)}
  \xrightarrow{\gamma_2=2i\circ \nu} (\KMil_{\star}/12) \{i\circ \nu\}\hookrightarrow
  \pistablesheaf_{2-(\star-3)}\PP^2\]
  is generated by $\eta$ and $6$ as a $\KMW$-module.
  In particular, in the relevant degree, the image
  of $\pia_{3}\SL_4\to \pia_3S^{3+(4)}$ is contained
  in the subsheaf $\eta\KMW_5+6\KMW_4$. Since it also contains
  this subsheaf, equality follows. This completes the proof.
\end{proof}

\begin{corollary}\label{cor:suslins-conjecture}
  Let $F$ be an infinite field with characteristic coprime to $6$. The inclusion
  $\BSL_3\hookrightarrow \BSL_\infty$ induces
  the following short exact sequence of sheaves:
  \[ 0 \to \KMil_4/6 \to \pia_3\BSL_3 \to \mathbf{K}^{\mathrm{Quillen}}_3 \to 0 \]
\end{corollary}

\begin{proof}
  This is a single case of \cite[Theorem 1.1]{af.bundles-spheres},
  having identified $\sheafS$ with $\KMil_4/6$ in
  the proof of Theorem~\ref{thm:suslin-conj}.
\end{proof}

\appendix 

\section{Modules over Milnor-Witt $K$-Theory}
\label{sec:modules-over-milnor}

Let $F$ be a field, usually suppressed from the notation.
The Milnor-Witt $K$-theory of $F$, as defined by Hopkins-Morel \cite{morel.zeroline},
is the unital graded associative
algebra generated by $\eta\in \KMW_{-1}\iso \pi_{1,1}\unit$ and 
the symbols $[u]\in \KMW_1(F)\iso \pi_{-1,-1}\unit_F$ for every unit $u\in F^\times$,
subject to the following relations.
\begin{description}
\item[Steinberg relation] $[u][v] = 0$ whenever $u+v =1$.
\item[$\eta$-twisted logarithm] $[uv]=[u]+[v]+\eta[u][v]$
\item[Commutativity] $[u]\eta=\eta[u]$
\item[Hyperbolic plane] $\eta+\eta^2[-1]=-\eta$
\end{description}
If $u_1,\dotsc,u_m\in F$ are units, then the element
$(1+\eta[u_1])+\dotsm +(1+\eta[u_m])\in \KMW_0(F)$
corresponds to the
quadratic form $\langle u_1,\dotsc,u_m\rangle$
given by the appropriate diagonal matrix under the
identification of $\KMW_{0}(F)$ with the Grothendieck-Witt ring $\GW(F)$
of $F$. Milnor 
$K$-theory \cite{milnor.k-quadratic} is expressed
as the quotient $\KMil_\star \iso \KMW_\star/\eta\KMW_\star$.
Often the grading will be suppressed from the notation.
Its reduction modulo 2 is denoted $\kmil$ for brevity.
If $A$ is a (graded) $\KMW$-module,
its degree $d$ part is $A_d$. For any $x\in \KMW_d$, the 
kernel and cokernel of multiplication with $x$ on $A_\ell$ are
denoted ${}_xA_\ell$ and $A_{\ell}/xA_{\ell-d}$. 
As a warm-up, $\KMil$-modules will be treated first.

\begin{theorem}\label{thm:ker-2-kmil}
  There is an equality $\{-1\}\KMil = {}_{2}\KMil$ of $\KMil$-modules.
\end{theorem}

\begin{proof}
  This follows as an equality of zero modules
  from \cite{izhboldin.torsion} if $F$ has characteristic 2.
  Suppose now that the characteristic of $F$ is different from 2.
  Since $2\{-1\}=\{1\} = 0\in \KMil_1$, there is an inclusion
  $\{-1\}\KMil \subset {}_{2}\KMil$ of $\KMil$-modules.
  The cofiber sequence
  \[ \MZ\xrightarrow{2}\MZ\xrightarrow{\pr^\infty_2} \MZ/2
    \xrightarrow{\partial_\infty^2} \Sigma\MZ \]
  induces a short exact sequence
  \begin{equation}\label{eq:ses-mzmod2}
    0 \to \pi_{1+\bideg}\MZ/2\pi_{1+\bideg}\MZ \to \pi_{1+\bideg}\MZ/2
    \xrightarrow{\partial^2_\infty} {}_{2}\pi_{0+\bideg}\MZ \to 0
  \end{equation}
  of $\KMil$-modules. Let $\tau \in \pi_{1-(1)}\MZ/2 = h^{0,1}$ denote
  the class of $-1\in F$.
  Then $\partial^2_\infty(\tau)=\{-1\}\in \pi_{0-(1)}\MZ = \KMil_1$
  by inspecting the short exact sequence~(\ref{eq:ses-mzmod2}) in the given
  weight, using that $-1$ is the unique nontrivial unit of order two in $F$.
  Voevodsky's solution of the Milnor conjecture on Galois cohomology
  \cite{voevodsky.mz2}
  (or rather the Rost-Voevodsky solution of the Beilinson-Lichtenbaum
  conjecture relating motivic and \'etale cohomology with coefficients in $\ZZ/2$)
  provides that multiplication with $\tau$ is an
  isomorphism $\pi_{0+\bideg}\MZ/2 \iso \pi_{1+\bideg}\MZ/2$
  of $\KMil$-modules. It follows that multiplication with $\{-1\}$ on
  $\KMil$ factors as
  \[ \KMil =\pi_{0+\bideg}\MZ\xrightarrow{\pr^\infty_2}
    \pi_{0+\bideg}\MZ/2 \xrightarrow{\tau \cdot}
    \pi_{1+\bideg}\MZ/2 \xrightarrow{\partial^2_\infty} {}_{2}\pi_{0+\bideg}\MZ
    ={}_{2}\KMil \subset \KMil \]
  in which the first three maps are surjective. The result follows.
\end{proof}

\begin{lemma}\label{lem:extkm2}
  Let $A$ be a $\KMil$-module.
  There is an isomorphism 
  \[\Ext^1_{\KMil}(\kmil,A)\iso \!\!{\ }_{\{-1\}}A_0/2A_0\]
  which is natural in $A$.
\end{lemma}

\begin{proof}
  The short exact sequence 
  \[ 0 \to 2\KMil \to \KMil \to \kmil \to 0\]
  induces an exact sequence
  \[ 0 \to \Hom_{\KMil}(\kmil,A)\to
  \Hom_{\KMil}(\KMil,A) \to
  \Hom_{\KMil}(2\KMil,A) \to
  \Ext^1_{\KMil}(\kmil,A)\to 0. \]
  The short exact sequence 
  \[ 0\to {}_2\KMil \to \KMil \to 2\KMil \to 0 \]
  induces an exact sequence
  \[ 0 \to \Hom_{\KMil}(2\KMil,A)\to
  \Hom_{\KMil}(\KMil,A) \to
  \Hom_{\KMil}({}_2\KMil,A) \to
  \Ext^1_{\KMil}(2\KMil,A)\to 0. \]
  The composition $\KMil \to 2\KMil \to \KMil$ coincides with multiplication
  by $2$, whence naturally $\Hom_{\KMil}(\kmil,A) = {}_2A_0$.
  In order to identify $\Hom_{\KMil}(2\KMil,A)$ in a similar way,
  observe the equality 
  \[ {}_2\KMil = \{-1\}\KMil \]
  of $\KMil$-modules from Theorem~\ref{thm:ker-2-kmil}.
  The composition $\KMil \to \{-1\}\KMil = {}_2\KMil\to \KMil$ 
  coincides with multiplication
  by $\{-1\}$, whence naturally $\Hom_{\KMil}(2\KMil,A) = {}_{\{-1\}}A_0$.
  The statement follows.
\end{proof}

\begin{lemma}\label{lem:extkm}
  Let $A$ be a $\KMW$-module.
  There is a natural isomorphism 
  \[\Ext^1_{\KMW}(\KMil,A)\iso {\ }_{\hyper}A_{-1}/\eta A_0.\]
\end{lemma}

\begin{proof}
  Abbreviate $\Hom(A,B):=\Hom_{\KMW}(A,B)$.
  The short exact sequence 
  \[ 0 \to \eta\KMW \to \KMW \to \KMil \to 0\]
  induces an exact sequence
  \[ 0 \to \Hom(\KMil,A)\to
  \Hom(\KMW,A) \to
  \Hom(\eta\KMW,A) \to
  \Ext^1_{\KMW}(\KMil,A)\to 0. \]
  The short exact sequence 
  \[ 0\to {}_\eta\KMW_{\star+1} \to \KMW_{\star+1} \to \eta\KMW \to 0 \]
  induces an exact sequence
  \[ 0 \to \Hom(\eta\KMW,A)\to
  \Hom(\KMW_{\star+1},A) \to
  \Hom({}_\eta\KMW_{\star+1},A) \to
  \Ext^1_{\KMW}(\eta\KMW,A)\to 0. \]
  The composition $\KMW_{\star+1} \to \eta\KMW \to \KMW$ 
  coincides with multiplication
  by $\eta$, whence naturally $\Hom(\KMil,A) = {}_{\eta}A_0$.
  In order to identify $\Hom(\eta\KMW,A)$ in a similar way,
  observe the existence of a further short exact sequence
  \[ 0\to {}_{\hyper}\KMW\to \KMW \to \hyper\KMW={}_\eta\KMW \to 0 \]
  of $\KMW$-modules, the surjection being induced by multiplication
  with $\hyper$. Note that $\hyper\KMW = \hyper \KMil = {}_{\eta}\KMW$
  by Theorem~\ref{thm:ker-eta-pi0}.
  The composition $\KMW_{\star+1} \to  {}_{\eta}\KMW_{\star+1}\to \KMW_{\star+1}$ 
  coincides with multiplication
   by $\hyper$, whence naturally $\Hom_{\KMW}(\eta\KMW,A) = {}_{\hyper}A_{-1}$.
   The statement follows.
\end{proof}  

\begin{lemma}\label{lem:extkm2inKMW}
  Let $A$ be a $\KMW$-module.
  There is a natural exact sequence
  \[ 0\to {}_{\{-1\}}\bigl({}_{\eta}A_0\bigr)/2\bigl({}_{\eta}A_0\bigr) \to
    \Ext^1_{\KMW}(\kmil,A) 
    \to {}_{\hyper}A_{-1}/\eta A_0 \]
\end{lemma}

\begin{proof}
  Abbreviate $\Hom(A,B):=\Hom_{\KMW}(A,B)$ and
  $\Ext(A,B):=\Ext_{\KMW}(A,B)$. The short exact sequence 
  \[ 0 \to 2\KMil \to \KMil \to \kmil \to 0\]
  of $\KMW$-modules induces an exact sequence
  \[ 0 \to \Hom(\kmil\!,\!A)\to
  \Hom(\KMil\!,\!A) \to
  \Hom(2\KMil\!,\!A) \to
  \Ext^1(\kmil\!,\!A)\to 
  \Ext^1(\KMil\!,\!A)\to \dotsm . \]
  The natural map
  $\Hom_{\KMW}(B,A)\to \Hom_{\KMil}(B, {}_{\eta}A)$
  is an isomorphism whenever $\eta B = 0$.
  The result then follows from Lemma~\ref{lem:extkm2} and
  Lemma~\ref{lem:extkm}.
\end{proof}

\section{Toda brackets}
\label{sec:toda-brackets}

Consider three composable maps
\[ \D \xrightarrow{\gamma} \E \xrightarrow{\beta}\F \xrightarrow{\alpha} \GG \]
of pointed motivic spaces or motivic spectra (over any base scheme)
such that the compositions $\beta\circ \gamma$ and $\alpha\circ \beta$ are
nullhomotopic.
Choosing nullhomotopies $\cone(\D):= \D\smash (\Delta^1,1)\to \F$
and $\cone(\E)=\E\smash (\Delta^1,1)\to \GG$
defines extensions of these compositions displayed
in the diagram
\begin{equation*}
  \begin{tikzcd}
    & \cone(\D) \ar[rd] & & \\
    \D \ar[ru] \ar[r,"\gamma"] & \E \ar[r,"\beta"] \ar[rd] & \F \ar[r,"\alpha"] & \GG \\
    & & \cone(\E) \ar[ru] &
  \end{tikzcd}
\end{equation*}
and hence a map $\Sigma \D \simeq \cone(\D)\cup_{\D}\cone(\E) \to \GG$,
where the identification of $\Sigma \D$ with the displayed pushout is
canonical. 
The {\em Toda bracket\/} 
\[ \langle \alpha,\beta,\gamma\rangle \subset [\Sigma\D,\GG] \]
is the subset of (pointed) homotopy classes maps which can be obtained this way.
It depends only on the homotopy classes of $\alpha,\beta,\gamma$.
This definition provides a secondary composition on the level
of homotopy categories and is sufficiently general to apply in manifold cases,
for example in any pointed simplicial model category.
In case $[\Sigma\D,\GG]$ is abelian (which is
always the case for motivic spectra), then
$\langle \alpha,\beta,\gamma\rangle$ is a coset of the
subgroup $\alpha\circ [\Sigma\D,\F]+[\Sigma\E,\GG] \circ \Sigma\gamma
\subset [\Sigma\D,\GG]$ \cite[Lemma 1.1]{toda}.
An equivalent description of $\langle \alpha,\beta,\gamma\rangle$
is obtained by taking 
homotopy classes of compositions
\[ \Sigma\D\xrightarrow{\gamma^\prime} \cone(\beta)=\cone(\E)\cup_{\E}\F
\xrightarrow{\widetilde{\alpha}} \GG\]
where $\widetilde{\alpha}$ is an extension of $\alpha$
to $\cone(\beta)$ and $\gamma^\prime$ is a coextension of $\gamma$ with respect
to $\beta$ \cite[Prop.~1.7]{toda}.
Any coextension of $\gamma$ with respect to $\beta$ is a lift of
$\Sigma\gamma$ along $\cone(\beta)\to \Sigma \E$.
In the stable case of motivic spectra, a coextension of $\gamma$
with respect to $\beta$ is the same
as a lift of $\Sigma\gamma$ along $\cone(\beta)\to \Sigma \E$
\cite[Prop.~5.2]{gray.sphere-origin}.
Two straightforward extensions of statements from \cite{toda}
will be stated explicitly for reference purposes.

\begin{proposition}\label{prop:toda-brackets-coext}
  Let three composable maps
  \[ \D \xrightarrow{\gamma} \E \xrightarrow{\beta}\F \xrightarrow{\alpha} \GG \]
  of pointed motivic spaces or motivic spectra (over any base scheme)
  be given such that the compositions $\beta\circ \gamma$ and
  $\alpha\circ \beta$ are nullhomotopic. Let
  $\GG\xrightarrow{i} \cone(\alpha)\xrightarrow{q} \Sigma \F$
  denote the canonical maps. There is an equality
  \[ - i\circ \langle \alpha,\beta,\gamma\rangle
  = \{ \beta^\prime\circ \Sigma\gamma \colon
  \beta^\prime \ \mathrm{coextension}\ \mathrm{of}\ \beta\ \mathrm{w.r.t.}\ \alpha\}\]
  of subsets of $[\Sigma\D,\GG]$.
\end{proposition}

\begin{proof}
  This follows by translating \cite[Proposition 1.8]{toda} to
  the motivic setting.
\end{proof}

\begin{proposition}\label{prop:toda-brackets}
  Let three composable maps
  \[ \D \xrightarrow{\gamma} \E \xrightarrow{\beta}\F \xrightarrow{\alpha} \GG \]
  of pointed motivic spaces or motivic spectra (over any base scheme)
  be given such that the compositions $\beta\circ \gamma$ and
  $\alpha\circ \beta$ are nullhomotopic. Then for every
  $\lambda\in \langle \alpha,\beta,\gamma\rangle$ there exists an
  extension $\widetilde{\beta}\colon \cone(\gamma)\to \F$ of $\beta$
  such that the diagram
  \begin{equation}\label{eq:toda-bracket}
    \begin{tikzcd}
      \cone(\gamma)\ar[r,"{\mathrm{can.}}"] \ar[d,"\widetilde{\beta}"] &
      \Sigma\D\ar[d,"\lambda"] \\
      \F \ar[r,"\alpha"] & \GG
    \end{tikzcd}
  \end{equation}
  commutes up to homotopy. Conversely, for every 
  extension $\widetilde{\beta}\colon \cone(\gamma)\to \F$ of $\beta$
  there exists an element
  $\lambda\in \langle \alpha,\beta,\gamma\rangle$ such that 
  diagram~(\ref{eq:toda-bracket}) commutes up to homotopy.
\end{proposition}

\begin{proof}
  This follows by translating \cite[Proposition 1.9]{toda} to
  the motivic setting.
\end{proof}

By adjointness, an
equivalent description in terms of maps to loop spaces can be given.
From the definition it follows quite immediately that any
functor $\Phi$ induced by a pointed simplicial left Quillen functor on
homotopy categories preserves Toda brackets in the sense of providing
an inclusion
\[ \Phi(\langle \alpha,\beta,\gamma\rangle) \subset
\langle \Phi(\alpha),\Phi(\beta),\Phi(\gamma)\rangle \]
whenever the Toda bracket $\langle \alpha,\beta,\gamma\rangle$
is defined. In particular, if the target Toda bracket
$\langle\Phi(\alpha),\Phi(\beta),\Phi(\gamma)\rangle$ does not contain
the zero element, neither does the source Toda bracket
$\langle\alpha,\beta,\gamma\rangle$. This applies in particular
to base change functors, stabilization, and to complex and \'etale realization.

For example,
the Toda bracket $\langle 2,\eta_\Top,2\rangle\subset \pi_5S^3$
in the pointed homotopy category of topological spaces defined by
\[ S^4\xrightarrow{2} S^4\xrightarrow{\eta_\Top} S^3\xrightarrow{2} S^3 \]
consists of the single element $\eta_\Top^2\colon S^5\to S^3$
by \cite[Lemma 5.2 and (5.2)]{toda}. 
With the help of the unstable considerations over
$\Spec(\ZZ)$ in \cite{dugger-isaksen.hopf}
(see in particular \cite[Remark A.10]{dugger-isaksen.hopf}),
the maps
\[ S^{2+(2)} \xrightarrow{\hyper} S^{2+(2)}\xrightarrow{\eta} S^{2+(1)}
\xrightarrow{\hyper} S^{2+(1)} \]
define a Toda bracket
$\langle\hyper,\eta,\hyper\rangle \subset [S^{3+(2)},S^{2+(1)}]$.
Comparison with the
$\PP^1$-stable situation and \cite[Theorem 5.5]{rso.oneline}
shows that any element in this Toda bracket maps to
$\eta\circ \eta_\Top\in \pi_{1+(1)}\unit_F$ modulo the subgroup
$2(\KMil_1(F)/24)$ generated by $\nu$ over fields $F$
of characteristic not two.
Similarly, any element in the Toda bracket
\[ \langle \eta,\hyper,\eta \rangle \subset [S^{3+(3)},S^{2+(1)}] \]
maps to $\pm 6\nu \in \pi_{1+(2)}\unit_F$ over fields $F$
of characteristic not two or three, by \cite[Theorem 5.5]{rso.oneline}
and \cite[(5.4), Prop.~5.6 and its proof]{toda}.
See \cite[Prop.~4.1]{rondigs.moore} for these two and further examples.

{\bf Acknowledgements.}
The author cordially thanks Aravind Asok, Jean Fasel and Ben Williams for their
helpful comments, as well as several
pleasant discussions on \cite{afw}. The anonymous referee provided many valuable
comments and corrections for which the author is sincerely grateful.
The author  would also like to thank the Isaac Newton Institute for Mathematical Sciences, Cambridge, for support and hospitality during the programme ``$K$-theory, algebraic cycles and motivic homotopy theory'' where the revision of this paper was undertaken. This work was supported by EPSRC grant no EP/R014604/1.
Finally, I thankfully acknowledge support by the DFG priority programme 1786 "Homotopy theory and algebraic geometry", the DFG project ``Algebraic bordism spectra'', and the RCN Frontier Research Group Project no.~312472 "Equations in motivic homotopy".
The hospitality of the Universitetet i Oslo is much appreciated.

\end{document}